\documentclass{amsart}
\usepackage{amssymb}
\usepackage{bbm}
\usepackage{hyperref}

\newcommand{\R}{\mathbbm R}

\newcommand{\N}{\mathbbm N}

\newcommand{\Htwo}{H^2(M;g_0)}
\newcommand{\Hfour}{H^4(M;g_0)}
\newcommand{\Div}{\mbox{div}}

\newtheorem{theorem}{Theorem}
\newtheorem{proposition}[theorem]{Proposition}
\newtheorem{lemma}[theorem]{Lemma}

\theoremstyle{definition}

\theoremstyle{remark}
\newtheorem{remark}[theorem]{Remark}

\numberwithin{equation}{section}
\numberwithin{theorem}{section}

\DeclareMathOperator{\vol}{Vol}

\newcommand{\coloneqq}{\,\raise0.08ex\hbox{\textnormal{:}}\!\!=}

\def\XXint#1#2#3{{\setbox0=\hbox{$#1{#2#3}{\int}$}
     \vcenter{\hbox{$#2#3$}}\kern-.5\wd0}}

\begin{document}

\title[``Large'' conformal metrics of prescribed $Q$-curvature]
{``Large'' conformal metrics of prescribed $Q$-curvature in the negative case}

\author{Luca Galimberti}
\address[Luca Galimberti]{Departement Mathematik\\ETH-Z\"urich\\CH-8092 Z\"urich}
\email{luca.galimberti@math.ethz.ch}
\thanks{Supported by SNF grant 200021\_140467 / 1. }
\date{\today}

\begin{abstract} 
Given a compact and connected four dimensional smooth Riemannian manifold $(M,g_0)$ with $k_P := \int_M  Q_{g_0} dV_{g_0}  <0$ and a smooth non-constant function $f_0$ with $\max_{p\in M}f_0(p)=0$, all of whose maximum points are non-degenerate, we assume that the Paneitz operator is nonnegative and with kernel consisting of constants. Then, we are able to prove that for sufficiently small $\lambda>0$ there are at least two distinct conformal metrics $g_\lambda=e^{2u_\lambda}g_0$ and $g^\lambda=e^{2u^\lambda}g_0$ of $Q$-curvature $Q_{g_\lambda}=Q_{g^\lambda}=f_0+\lambda$. Moreover, by means of the ``monotonicity trick'' in a way similar to \cite{Borer-Galimberti-Struwe}, we obtain crucial estimates for the ``large'' solutions $u^\lambda$ which enable us to study their ``bubbling behavior'' as $\lambda \downarrow 0$.
\end{abstract} 

\maketitle

\section{Introduction} 

Given a smooth Riemannian manifold $(M,g_0)$ and a function $f:M\to\R$, an important problem in conformal geometry is to find conditions on $f$ in order that it arises as a certain kind of curvature of
a metric $g$ conformal to $g_0$. In dimension 2, one usually considers the Gauss curvature and is led to the
classical problem of prescribing the Gauss curvature. We refer the reader to the classical references
\cite{Berger71},\cite{Kazdan-Warner74} and \cite{Borer-Galimberti-Struwe} for a recent review
of the state of the art for this problem.

In dimension 4, a natural curvature to be considered is the $Q$-curvature, introduced in \cite{Branson85}
and associated to the Paneitz operator, a conformally invariant operator which first appeared in 
\cite{Paneitz82}. More precisely, let $(M,g_0)$ be a closed and connected 4-dimensional Riemannian manifold endowed with a smooth background metric $g_0$. The $Q$-curvature $Q_{g_0}$ and the Paneitz operator $P_{g_0}$
are defined in terms of the Ricci tensor $\mbox{Ric}_{g_0}$ and the scalar curvature $R_{g_0}$ of $(M,g_0)$ as
\begin{equation}\label{eqn: definizione Q-curvature}
Q_{g_0}=-\frac{1}{12}(\Delta_{g_0}R_{g_0}-R_{g_0}^2+3|\mbox{Ric}_{g_0}|^2),  
\end{equation}
\begin{equation}\label{eqn: definizione Paneitz}
P_{g_0}(\varphi) =\Delta^2_{g_0}\varphi - \Div_{g_0}\left(\frac{2}{3} R_{g_0}g_0 -2\mbox{Ric}_{g_0} \right)d\varphi  
\end{equation}
where $\Delta_{g_0}$ is the Laplace Beltrami of $(M,g_0)$ and $\varphi$ is any smooth function on $M$. The relation between $P_{g_0}$ and $Q_{g_0}$, when one performs
a conformal change of metric $g=e^{2u}g_0$, is given by
\begin{equation}\label{eqn: cambio conforme di metrica}
P_{g}=e^{-4u}P_{g_0}; \;\;\;\;\;\;\;\; P_{g_0}u + 2Q_{g_0} = 2Q_{g}e^{4u},
\end{equation}
which may be viewed as the analogue of the transformation rule for Gauss curvature in dimension 2. Moreover,
one has the following extension of the Gauss-Bonnet formula (compare \cite{Branson05})
\begin{equation}\label{eqn: Gauss Bonnet}
\int_M \left( Q_{g_0}+\frac{|W_{g_0}|^2}{8}\right) dV_{g_0} = 4\pi^2\chi(M),  
\end{equation}
where $\chi(M)$ is the Euler characteristic of $M$ and $W_{g_0}$ denotes the Weyl tensor of $(M,g_0)$. From
the pointwise conformal invariance of $|W_{g_0}|^2 dV_{g_0}$, it readily follows that also the quantity
\begin{equation}\label{eqn: k_P}
k_P := \int_M  Q_{g_0} dV_{g_0}  
\end{equation}
is a conformal invariant.

Hence, our initial problem can be stated as follows: given a function $f:M\to\R$, we look for conditions on $f$ such that the equation
\begin{equation}\label{eqn: cambio conforme di metrica (2)}
P_{g_0}u + 2Q_{g_0} = 2fe^{4u}  
\end{equation}
admits a solution. In view of the conformal invariance of (\ref{eqn: k_P}), we immediately deduce a first set of necessary conditions on $f$ for the solvability of (\ref{eqn: cambio conforme di metrica (2)}), depending on the sign of $k_P$. More precisely, if $k_P>0$, then $f$ must be positive somewhere; if $k_P<0$, then
$f$ must be negative somewhere; if $k_p=0$, then $f$ must change sign or must be identically zero.

In the case of the standard sphere $S^4$, Wei and Xu \cite{Wei-Xu98} showed that (\ref{eqn: cambio conforme di metrica (2)}) admits a solution when the prescribed function $f$ is positive, and under some conditions involving the critical points of $f$ and the topological degree of a certain map defined in terms of $f$. Later, Brendle 
\cite{Brendle03} was able to construct conformal metrics whose $Q$-curvature is a constant multiple of a prescribed positive function on a general $M$, under the assumptions that the Paneitz operator is nonnegative with kernel 
consisting of only constant functions and $k_P<8\pi^2$. As a consequence, he was able to generalize Moser's theorem
for prescribed Gauss curvature on the projective plane to dimension $n$. In \cite{Baird-Fardoun-Regbaoui}, Baird, Fardoun and Regbaoui,
constructing a suitable gradient flow, were able to give new sufficient conditions on $f$, depending on the
sign of $k_P$, in order that (\ref{eqn: cambio conforme di metrica (2)}) admits a solution, as soon as one assumes the nonnegativity of the Paneitz operator and that its kernel consists of constant functions only. In particular, in the
negative case, they could prove existence of solutions to (\ref{eqn: cambio conforme di metrica (2)}) for
functions $f$ changing sign and not ``too'' positive.

The afore-mentioned existence results give almost no information about the structure of the set of solutions to equation (\ref{eqn: cambio conforme di metrica (2)}). Goal of this paper is try to shed some light on the
set of solutions and its compactness properties. We will focus on the negative case, viz $k_P<0$, and we will assume that the Paneitz operator is nonnegative with kernel consisting uniquely of constants. Note that Eastwood and Singer \cite{Eastwood-Singer93} constructed metrics on connected sums of $S^3 \times S^1$ with $k_P<0$, $P_{g_0}\geq 0$ and kernel consisting of constants functions. Under these assumptions, the analogue of the uniformization theorem holds 
(we refer the reader for instance to \cite{Chang-Yang95} and \cite{Djadli-Malchiodi08}); thus we can assume
that $M$ carries a background metric $g_0$ such that $Q_{g_0}=\mbox{const}<0$. Finally, by convenience, we normalize the volume of $(M,g_0)$ to unity. Therefore, 
\[
Q_{g_0}=k_P<0\,.
\]

In this setting and in complete analogy to the case of surfaces of higher genus (compare \cite{Aubin70}), solutions of (\ref{eqn: cambio conforme di metrica (2)}) can be characterized as critical points
of the following energy
\begin{equation}\label{eqn: the energy}
E_f(u)= \langle P_{g_0}u,u \rangle +4Q_{g_0}\int_M u\, dV_{g_0} - \int_M fe^{4u}dV_{g_0}, \;\; u\in\Htwo\,,  
\end{equation}
where 
\[
\langle P_{g_0}u,v\rangle = \int_M \left[\Delta_{g_0}u \Delta_{g_0}v +\frac{2}{3}R_{g_0} g_0\left(\nabla_{g_0}u,\nabla_{g_0}v\right) -2\mbox{Ric}_{g_0}\left(\nabla_{g_0}u,\nabla_{g_0}v\right) \right] dV_{g_0},  
\]
with $u,v\in\Htwo$. Note that in view of Adams' inequality \cite{Adams88}, the above energy is well defined on $\Htwo$. We also remark that, if $f\in C^{\infty}(M)$, by standard regularity arguments
(see for instance Thm 7.1 \cite{Agmon59}) it follows that critical points of $E_f$ are of class $C^{\infty}$ and hence are classical solutions of (\ref{eqn: cambio conforme di metrica (2)}).

Our first result is:
\begin{theorem}\label{thm: metodo diretto}
Let $(M,g_0)$ be closed and connected with $k_P<0$, $P_{g_0}\geq 0$ and $\ker(P_{g_0})=\left\{constants\right\}.$ Let $0\neq f\in C^0(M)$ with $f\leq 0$. Then (\ref{eqn: cambio conforme di metrica (2)}) admits a unique solution in $\Hfour$.  
\end{theorem}
The unique solution is the absolute minimizer of $E_f$, which is strictly convex and coercive if $f\leq 0$ (see Lemma \ref{lemma: E_f is strictly convex}). From this theorem and following an idea by \cite{Bismuth2000}, we can obtain a stability result for equation (\ref{eqn: cambio conforme di metrica (2)}), which guarantees the existence of relative minimizers for the energy $E_f$, even when $f$ changes sign.
\begin{theorem}\label{thm: stability result}
Let $(M,g_0)$ be closed and connected with $k_P<0$, $P_{g_0}\geq 0$ and $\ker(P_{g_0})=\left\{constants\right\}.$ Suppose $0\neq f\in C^{0,\alpha}(M)$ for some $\alpha\in(0,1)$ and with $f\leq 0$. Then there exists 
$\mathcal{N}\subset C^{0,\alpha}(M)$ open neighborhood of $f$ such that for all $h\in\mathcal{N}$ there exists
a strict relative minimizer for $E_h$ in $C^{4,\alpha}(M)$ smoothly dependent on $h$. In particular, if $f$ and $h$ are in $C^{\infty}(M)$, then the minimizer is in $C^{\infty}(M)$ as well.  
\end{theorem}
In particular, we recover Theorem 2.6 of \cite{Baird-Fardoun-Regbaoui}.
We then consider a nonconstant smooth function $f_0\leq 0$ with $\max_{p\in M}f_0(p)=0$, all of whose maximum points are non-degenerate. We set $f_\lambda:=f_0+\lambda$, where $\lambda\in\R$. From Thm \ref{thm: stability result} we deduce the existence of a strict relative minimizers $u_\lambda\in C^{\infty}(M)$ of $E_\lambda:=E_{f_\lambda}$ for all sufficiently small $\lambda>0$, where $u_\lambda$ solves the equation
\begin{equation}\label{eqn: equazione target}
P_{g_0}u_{\lambda} + 2Q_{g_0} = 2f_{\lambda}e^{4u_\lambda} .  
\end{equation}
We observe that for functions $f$ with $\max_M f>0$ we have $\inf_{\Htwo}E_f=-\infty$. Indeed, choosing $w\in C^\infty(M)$, $0\leq w\leq 1$ and with support in the set $\left\{f>0\right\}$, then one
has $\lim_{t\to +\infty}E_f(tw)=-\infty$. Therefore, since for $\lambda>0$ sufficiently small $E_\lambda$ admits
a relative minimizer, we observe the presence of a ``mountain pass'' geometry and the intuition would suggest the existence of a further critical point, if we could guarantee some compactness properties.
In Proposition \eqref{prop: Palais-Smale}, we indeed prove that for a generic $f\in C^2(M)$ the functional $E_f$ possesses bounded Palais-Smale sequences at any level $\beta\in\R$. This fact enables one to conclude that for all sufficiently small $\lambda>0$ the functional $E_\lambda$ admits, in addition to a strict relative minimizer $u_\lambda$, a further critical point $u^\lambda\neq u_\lambda$ of mountain pass type.

However, this abstract result gives no additional information at all about how the ``limit'' geometry of the manifolds $(M,e^{2u^\lambda}g_0)$ could look like when $\lambda\downarrow 0$. In order to answer to this issue, firstly, we employ Struwe's ``monotonicity trick'' in a way similar to \cite{Borer-Galimberti-Struwe} to obtain a suitable sequence of ``large'' solutions $u^\lambda$. Secondly, thanks to an appropriate choice of a comparison function for our ``mountain pass'' geometry, we derive some refined estimates which enables us to bound the volume of these second solutions and to prove Theorem \ref{thm: second solution} and Theorem \ref{thm: bubbling analysis}. More precisely, we have:
\begin{theorem}\label{thm: second solution}
Let $(M,g_0)$ be closed and connected with $k_P<0$, $P_{g_0}\geq 0$ and $\ker(P_{g_0})=\left\{constants\right\},$
and consider any smooth, nonconstant function $f_0\leq 0=\max_{p\in M}f_0(p)$, all of whose maximum points are non-degenerate. Consider the family of functions $f_\lambda=f_0+\lambda,\,\lambda\in\R,$ and the associated family of functionals $E_\lambda(u)=E_{f_\lambda}(u)$, $u\in\Htwo$. There exists a number
$\lambda^\ast>0$ such that for almost every $0<\lambda<\lambda^\ast$ the functional $E_{\lambda}$ admits a strict relative minimizer $u_{\lambda}$ and a further critical point $u^{\lambda}\neq u_{\lambda}$.  
\end{theorem}
The picture obtained by combining Theorem \ref{thm: stability result} and Theorem \ref{thm: second solution} reminds us of a two-branches bifurcation diagram, with a branch consisting of relative minimizers $u_\lambda$ smoothly converging to the unique solution of (\ref{eqn: cambio conforme di metrica (2)}) when $\lambda \downarrow 0$, and a second ``branch'' (defined a.e.) consisting of the ``large'' solutions $u^{\lambda}$. 
\begin{theorem}\label{thm: bubbling analysis}
Under the hypothesis of Theorem \ref{thm: second solution}, there exist a sequence $\lambda_n \downarrow 0$, a sequence of solutions $(u_n)_n$ of the equation 
\[
P_{g_0}u_n + 2Q_{g_0} = 2f_{\lambda_n}e^{4u_n}
\]
and there exists $I\in\N$ such that, for suitable $p_n^{(i)}\to p_\infty^{(i)}\in M$ with $f_0(p_\infty^{(i)})=0,1\leq i\leq I$, we obtain $u_n(p_n^{(i)})\to\infty$ and one of the following: either
\begin{enumerate}
  \item a subsequence $(u_n)_n$ converges locally uniformly to $-\infty$ on $M_\infty:=M\setminus \left\{p_\infty^{(1)},\cdots ,p_\infty^{(I)} \right\}$, or
  \item a subsequence $(u_n)_n$ converges locally smoothly on $M_\infty$ to $u_\infty$, which induces on $M_\infty$ a metric $g_\infty=e^{2u_\infty}g_0$ of finite total $Q$-curvature $Q_{g_\infty}=f_0$.
\end{enumerate}
In any case, for each $1\leq i\leq I$ and for suitable $r_n^{(i)}\downarrow 0$, either

a) $r_n^{(i)}/\sqrt{\lambda_n}\to 0$ and in normal coordinates around $p_\infty^{(i)}$ we have, by setting $x_n=\exp^{-1}(p_n^{(i)})$ and $\tilde{u}_n=u_n\circ\exp$, that
\[
\hat{u}_n(x):= \tilde{u}_n(x_n+r_n^{(i)}x)-\tilde{u}_n(x_n)\to w(x)=-\log\left(1+\frac{|x|^2}{4\sqrt{6}}\right)
\] 
in $C^{4,\alpha}_{loc}(\R^4)$, where $w$ induces a spherical metric $g=e^{2w}g_{\R^4}$ of $Q$-curvature $Q_g=1/2$ on $\R^4$, or

b) $r_n^{(i)}=c^{(i)}\sqrt{\lambda_n}$ for a suitable constant $c^{(i)}>0$ and in normal coordinates around $p_\infty^{(i)}$ we have, by setting $\tilde{u}_n=u_n\circ\exp$, that
\[
\hat{u}_n(x):= \tilde{u}_n(r_n^{(i)}x) +\frac{3}{4}\log(\lambda_n)+\log(c^{(i)})\to w(x)
\] 
in $C^{4,\alpha}_{loc}(\R^4)$, where the metric $g=e^{2w}g_{\R^4}$ on $\R^4$ has finite volume and finite total $Q$-curvature $Q_g(x)=1+\frac{1}{2}D^2f_0(p_\infty^{(i)})\left[x,x\right]$.

\end{theorem}

\section{Stability result}
In this section we prove Thm \ref{thm: metodo diretto} and Thm \ref{thm: stability result}. Throughout the rest
of the paper, the Paneitz operator is assumed to be nonnegative and with kernel consisting of constant functions. In view of these hypothesis, it is straightforward to see that there exists a constant $C\geq 1$ only
depending on $M$, such that for all $u\in\Htwo$
\begin{equation}\label{eqn: poincare per il laplaciano}
C^{-1}||\Delta_{g_0}u||^2_{L^2(M)} \leq \langle P_{g_0}u,u\rangle\leq C||\Delta_{g_0}u||^2_{L^2(M)} .
\end{equation}
As a consequence, the bilinear map
\begin{equation}\label{eqn: prodotto scalare equivalente}
\Htwo \times \Htwo \ni (u,v) \mapsto \langle P_{g_0}u,v\rangle + \int_M uv\, dV_{g_0} .
\end{equation}
defines an equivalent scalar product on $\Htwo$.
\begin{lemma}\label{lemma: E_f is coercive}
Let $(M,g_0)$ be closed and connected with $k_P<0$. Let $0\neq f\in C^0(M)$ with $f\leq 0$. Then the functional $E_f$ is coercive on $\Htwo$.
\end{lemma}
\begin{proof}
Since for any fixed $c>0$ there holds that
\[
E_f(u)=-Q_{g_0}\log c + E_{f/c}(u+(\log c)/4)\,, \;\; u\in\Htwo
\]
and moreover $||u||_{\Htwo}\to\infty$ iff $||u-\log c^{-1/4}||_{\Htwo}\to\infty$, we may assume that $||f||_{L^1(M)}=1\left(=\int_M -f\, dV_{g_0}\right)$. The general result will then follow by setting $c=||f||_{L^1(M)}>0$.

We define for any $u\in L^1(M)$ the $f$-average of $u$ as
\[
\bar{u}^f := \int_M -fu\, dV_{g_0} \,.
\]
Via Jensen's inequality, applied to the probability measure $-fdV_{g_0}$, we obtain for any
$u\in\Htwo$
\[
\int_M -fe^{4u}\, dV_{g_0} \geq \exp(4\bar{u}^f)
\]
and thus
\begin{equation}\label{eqn: E_f is coercive}
\begin{split}
E_f(u) & \geq \langle P_{g_0}u,u\rangle + 4Q_{g_0}\int_M u\, dV_{g_0} +   \exp(4\bar{u}^f)\\  
       & = \langle P_{g_0}u,u\rangle +  \left[4Q_{g_0}\bar{u}^f + \exp(4\bar{u}^f) \right]
          + 4Q_{g_0}\int_M u(1+f)\, dV_{g_0}\,.
\end{split} 
\end{equation}
We set $\bar{u}:= \int_M u\,dV_{g_0}$ (recall that $\vol(M,g_0)=1$) and note that
\begin{equation*}
\begin{split}
\left|\int_M u(1+f)\, dV_{g_0} \right| & = | \bar{u} - \bar{u}^f \,| \leq  ||u-\bar{u}||_{L^2(M)} + ||\bar{u}-\bar{u}^f\,||_{L^2(M)}\\
&  \leq \bar{C} ||\nabla u||_{L^2(M)}\leq \bar{C} ||\Delta u||_{L^2(M)}\,,
\end{split}
\end{equation*}
by a variant of the Poincar\'e inequality. 

Hence, by \eqref{eqn: E_f is coercive} and \eqref{eqn: poincare per il laplaciano},
\[
\begin{split}
E_f(u) &\geq \langle P_{g_0}u,u\rangle +  \left[4Q_{g_0}\bar{u}^f + \exp(4\bar{u}^f) \right]
-\bar{C} ||\Delta u||_{L^2(M)}\\
& \geq \frac{1}{2C} ||\Delta u||_{L^2(M)}^2 +  \left[4Q_{g_0}\bar{u}^f + \exp(4\bar{u}^f) \right],
\end{split}
\]
where the last inequality holds uniformly if $||u||_{\Htwo}>>0$.
Since $Q_{g_0}<0$, for suitable constants $a_1,a_2>0$ we have for every $t\in\R$ that 
\[
4Q_{g_0}t+e^{4t}\geq a_1|t|-a_2\,.
\] 
It also follows
\[
E_f(u) \geq \frac{1}{2C} ||\Delta u||_{L^2(M)}^2 +a_1|\bar{u}^f|-a_2 \,.
\]
Again by the variant of the Poincar\'e inequality, it also follows 
\[
||u||_{L^2(M)}\leq \bar{C}||\Delta u||_{L^2(M)} + |\bar{u}^f|
\]
and thus we obtain
\[
E_f(u)\geq \frac{1}{4C} ||\Delta u||_{L^2(M)}^2 + a_1||u||_{L^2(M)} -a_2  \,,
\]
for all $||u||_{\Htwo}$ sufficiently large. The claim follows.
\end{proof}

\begin{lemma}\label{lemma: E_f is strictly convex} 
Under the same hypothesis as in the previous lemma we have that for any $u\in\Htwo$ there exists a number $\nu=\nu(u)>0$ such that
\begin{equation}\label{eqn: non degenericita' dell'Hessiano}
\frac12D^2E_f(u)\left[w,w\right]\geq \nu \left[\langle P_{g_0}w,w\rangle+  ||w||^2_{L^2(M)}\right]  
\end{equation}
for any $w\in\Htwo$. 
\end{lemma}
\begin{proof}
For any $u,w\in\Htwo$, $w\neq 0$, we have
\begin{eqnarray*}
\frac12D^2E_f(u)\left[w,w\right]&=&\frac12\frac{d^2}{ds^2}\bigg|_{s=0} E_f(u+sw) \\
                  & =&  \langle P_{g_0}w,w\rangle -8 \int_M fe^{4u}w^2 dV_{g_0}\,.\\
\end{eqnarray*}
We suppose now that $0=\inf\left\{\langle P_{g_0}w,w\rangle -\int_M 8fe^{4u}w^2\, dV_{g_0}: w\in\Htwo\right\}$ for some $u\in\Htwo$. Hence, there exists a sequence $(w_n)_n$ such that $\langle P_{g_0}w_n,w_n\rangle+  ||w_n||^2_{L^2(M)}=1$, $\langle P_{g_0}w_n,w_n\rangle -\int_M 8fe^{4u}w_n^2\, dV_{g_0}\to 0$, $w_n\rightharpoonup w$ in $\Htwo$ and $w_n\to w$ in $L^2(M)$.

At once we see that $\langle P_{g_0}w_n,w_n\rangle\to 0$ and, by Poincar\'e inequality, that $w_n-\bar{w}_n\to 0$ in $L^2(M)$. Therefore, there holds $w\equiv c$ with $c=\pm 1$. But then $-\int_M 8fe^{4u}w_n^2\, dV_{g_0} \to  0=-\int_M 8fe^{4u}\, dV_{g_0}>0$. This contradiction proves the Lemma.
\end{proof}
In view of Lemma \ref{lemma: E_f is strictly convex} and \ref{lemma: E_f is coercive} we can apply the direct
method of the Calculus of Variations and see that $E_f$ admits a unique absolute minimizer $u_f\in\Htwo$, which solves
\begin{equation}\label{eqn: formulazione debole}
\langle P_{g_0}u_f,v\rangle + 2Q_{g_0}\int_M v\, dV_{g_0} = 2\int_M fe^{4u_f}v\, dV_{g_0}  
\end{equation}
for any $v\in\Htwo$. By elliptic regularity theory we see that $u_f\in\Hfour$ and therefore Theorem \ref{thm: metodo diretto} follows. 

We note that, if we further impose $f\in C^{0,\alpha}(M)$, it follows by the embedding $\Hfour\subset
C^{1,\alpha}(M)$ and Schauder's estimates that $u_f\in C^{4,\alpha}(M)$. We have:
\begin{proposition}\label{prop: stability result 1}
Let $(M,g_0)$ be closed and connected with $k_P<0$. Let $0\neq f\in C^{0,\alpha}(M)$ for some
$0<\alpha<1$ and $f\leq 0$. Then there exist $\mathcal{N}\subset C^{0,\alpha}(M)$ open neighborhood of $f$ and $G\in C^1(\mathcal{N},C^{4,\alpha}(M))$ such that, for any $h\in\mathcal{N}$, $G(h)$ is a classical solution of
\[
P_{g_0}G(h) + 2Q_{g_0} = 2he^{4G(h)} . 
\] 
\end{proposition}
\begin{proof}
We consider the map
\[
Z: C^{4,\alpha}(M) \times C^{0,\alpha}(M) \to C^{0,\alpha}(M)
\]
\[
(w,h)\stackrel{Z}{\longmapsto} P_{g_0}w +2\left(f-he^{4w}  \right)e^{4u_f},
\]
where $u_f\in C^{4,\alpha}(M)$ is the unique solution of (\ref{eqn: cambio conforme di metrica (2)}), given by (\ref{eqn: formulazione debole}). $Z$ is clearly $C^1$, $Z(0,f)=0$ and it is straightforward to see that
\[
D_wZ(0,f)= P_{g_0}\cdot -8fe^{4u_f}\cdot \in \mathcal{L}(C^{4,\alpha}(M);C^{0,\alpha}(M))\,.
\]
From \ref{eqn: non degenericita' dell'Hessiano}, Lax-Milgram Theorem and standard elliptic regularity arguments, we infer
\[
D_wZ(0,f) \in \mathcal{I}nv(C^{4,\alpha}(M);C^{0,\alpha}(M))\,.
\]
We apply the Implicit Function Theorem and obtain an open neighborhood $\mathcal{N}\subset C^{0,\alpha}(M)$ of $f$, $\mathcal{M}_0\subset C^{4,\alpha}(M)$ open neighborhood of 0 and $G_0\in C^1(\mathcal{N},\mathcal{M}_0)$ such that for any $h\in\mathcal{N}$
\[
Z(G_0(h),h) = 0\,;
\]
that is $ P_{g_0}G_0(h) +2\left(f-he^{4G_0(h)}  \right)e^{4u_f} = 0$. Since $u_f$ solves (\ref{eqn: cambio conforme di metrica (2)}), we obtain 
\[
 P_{g_0}\left(G_0(h) + u_f\right) +2Q_{g_0} = 2he^{4(G_0(h)+u_f)}\,.
\]
Therefore, setting
\[
G(h):=G_0(h)+u_f\,,
\]
we obtain the desired conclusion.
\end{proof}
\begin{proof}[Proof of Thm \ref{thm: stability result}]
From equation (\ref{eqn: non degenericita' dell'Hessiano}) we see that, for all $w\in\Htwo$ with $||w||^2_{\Htwo} =1  $, there holds
\[
D^2E_f(u_f)\left[w,w\right] = 2\langle P_{g_0}w,w\rangle -16\int_M fe^{4u_f}w^2\, dV_{g_0}\geq 2\nu>0\,.
\]
For $h\in\mathcal{N}$ we consider $G(h)\in C^{4,\alpha}(M)$, defined as in the proof of Proposition \ref{prop: stability result 1}. Therefore $G(h)$ is a critical point
of the functional $E_h$, whose Hessian at the point $G(h)$ in the unit direction vector $w$ is
\[
D^2E_h(G(h))\left[w,w\right] = 2\langle P_{g_0}w,w\rangle -16\int_M he^{4G(h)}w^2\, dV_{g_0}\,.
\]
Since the map $G$ is $C^1$, we have $e^{4G(h)}\to e^{4u_f}$ in $C^{4,\alpha}(M)$ when $h\to f$ in $C^{0,\alpha}(M)$ and therefore
\[
\int_M \left(he^{4G(h)}w^2-fe^{4u_f}w^2\right)\, dV_{g_0} \to 0
\]
 when $h\to f$ in $C^{0,\alpha}(M)$, uniformly in $w$, $||w||_{\Htwo}=1$. Therefore,
by choosing a smaller neighborhood $\mathcal{N}$ of $f$ from the beginning and by homogeneity of the Hessian, we obtain
\[
D^2E_h(G(h))\left[w,w\right] > \nu ||w||^2_{\Htwo}
\]
for all $h\in\mathcal{N}$ and $w\in\Htwo$. Recalling that $G(h)$ is a critical point of $E_h$, by means of a 
second order Taylor expansion we conclude that $G(h)$ is a strict relative minimizer for $E_h$.
\end{proof}
\section{Existence of a second critical point}
Let $f_0\leq 0$ be a nonconstant smooth function with $\max_{p\in M}f_0(p)=0$, all of whose maximum points are non-degenerate. Set $f_\lambda:=f_0+\lambda$, $\lambda\in\R$, and consider $E_\lambda(u):=E_{f_\lambda}(u)$, $u\in\Htwo$. By Theorem \ref{thm: stability result} we deduce the existence
of a number $\lambda_0>0$ such that for any $\lambda\in\Lambda_0=\left(0,\lambda_0 \right]$ the functional
$E_\lambda$ admits a strict relative minimizer $u_\lambda\in C^\infty(M)$, depending smoothly on $\lambda$. In particular, calling $u_0$ the unique (smooth) solution of (\ref{eqn: cambio conforme di metrica (2)}) for $f=f_0$, we see that, as $\lambda \downarrow 0$, $u_\lambda\to u_0$ smoothly in $\Htwo$.
Hence, after replacing $\lambda_0$ with a smaller number $1/4>\lambda_0>0$, if necessary, 
we can find $\rho >0$ such that 
\begin{equation}\label{eqn: struttura mountain-pass}
 \begin{split}
  E_\lambda(u_\lambda) =\inf_{||u-u_0||_{H^2}<\rho}E_\lambda(u)
  &\le \sup_{\mu,\nu \in \Lambda_0} E_\mu(u_\nu)\\
  &< \beta_0:=\inf_{\mu\in\Lambda_0;\,\rho/2<||u-u_0||_{H^2}<\rho}E_{\mu}(u),
 \end{split}
\end{equation}
uniformly for all $\lambda\in\Lambda_0$. Fix some number $\lambda\in\Lambda_0$. Recalling that for $\lambda > 0$ the functional $E_\lambda$ is unbounded from below, it is also possible to fix a function $v_{\lambda}\in \Htwo$ 
such that
\[
E_{\lambda}(v_{\lambda})<E_{\lambda}(u_{\lambda})
\] 
and hence 
\begin{equation}\label{eqn: definizione di c_lambda}
 c_\lambda = \inf_{p\in P}\;\max_{t\in[0,1]} E_{\lambda}(p(t)) \ge 
 \beta_0> E_{\lambda}(u_{\lambda}),
\end{equation}
where 
\begin{equation}\label{eqn: l'insieme dei cammini}
  P = \{p\in C^0\left([0,1];  \Htwo \right):\; p(0)=u_0,\, p(1)= v_{\lambda} \}.
\end{equation}
Because $u_\lambda\to u_0$ smoothly in $\Htwo$ when $\lambda \downarrow 0$, it is possible to fix the initial point of the comparison paths $p\in P$ to be $u_0$ instead of $u_\lambda$, provided that $\lambda_0$ is sufficiently small.

For a suitable choice of $v_\lambda$, we obtain an explicit and useful estimate of the mountain-pass energy level
$c_\lambda$ associated with $P$.
\begin{proposition}\label{prop: stima logaritmica}
For any $K>32\pi^2$ there is $\lambda_K\in]0,\lambda_0/2]$ such that for any $0<\lambda<\lambda_K$ 
there is $v_{\lambda}\in \Htwo$ so that choosing $v_{\mu}=v_{\lambda}$ 
for every $\mu\in [\lambda,2\lambda]$ we obtain the bound $c_\mu\le K\log(1/\mu)$.
\end{proposition}
The proof of this proposition is largely inspired by \cite{Borer-Galimberti-Struwe}, even though in the present setting further complications arise, due to the fact that we are dealing with a differential operator of a higher order than the Laplacian as in \cite{Borer-Galimberti-Struwe}, and therefore one must handle a number of extra terms appearing in the estimates. In a nutshell, our strategy to construct a suitable comparison function $v_\lambda$ will consist in using an appropriate truncated and scaled version of a fundamental solution of the bi-Laplacian operator in $\R^4$.

The proof of the proposition is quite long and will be postponed after Proposition \ref{prop: esistenza di una seconda soluzione}. Therefore, we now proceed with the proof of the existence of a second critical point.

Note that for any $u\in\Htwo$ and  for every $\mu_1,\mu_2\in \R$ there holds
\begin{equation}\label{eqn: differenza delle energie}
   E_{\mu_1}(u)-E_{\mu_2}(u)= (\mu_2-\mu_1) \int_M e^{4u}\,dV_{g_0} \,.
\end{equation}
It follows that the function
\begin{equation}\label{eqn: la c mu e' non increasing}
  \Lambda \ni \mu \mapsto c_{\mu}
\end{equation}
is non-increasing in $\mu$, and therefore differentiable at almost every $\mu\in \Lambda$. 

We note the following lemma, which is the analogue of Lemma 3.3 in \cite{Borer-Galimberti-Struwe}:
\begin{lemma}\label{lemma: stime per E e DE}
i) For any $m>0$ there exists a constant $C=C(m)$ such that for every
$\mu_1,\mu_2\in\R$ and for every $u\in \Htwo$ satisfying $||u||_{\Htwo}\leq m$ there
holds
\begin{equation*}
 || DE_{\mu_1}(u) - DE_{\mu_2}(u)|| \leq C |\mu_1 - \mu_2|\, .
\end{equation*}   

ii) For any $|\mu|<1$, any $u,v\in \Htwo$ with $|| v ||_{\Htwo}\leq 1$, we
have
\begin{equation*}
  E_{\mu}(u+v)\leq E_{\mu}(u) + DE_{\mu}(u)[v] 
  + \left[1+C\left\{ \int_M e^{16u} \,dV_{g_0} \right\}^{1/4}\right]|| v ||_{\Htwo}^2\,,
\end{equation*}
where $C=C(M,g_0,f_0)$ is a positive constant.
\end{lemma}
\begin{proof}
i) Take $v\in\Htwo$ such that $||v||_{\Htwo}\leq 1$ and compute
\[
\begin{split}
DE_{\mu_1}&(u)[v]- DE_{\mu_2}(u)[v]= (\mu_2-\mu_1)\int_M e^{4u} v  \,dV_{g_0} \\
&\leq |\mu_2-\mu_1|\left\{\int_M e^{8u}\,dV_{g_0}\right\}^{1/2} ||v||_{L^2(M;g_0)} \leq  
|\mu_2-\mu_1|\left\{\int_M e^{8u}\,dV_{g_0}\right\}^{1/2} \,.
\end{split}
\]
The claim follows from Adams's inequality \cite{Adams88}.

ii) By Taylor's expansion, for every $x\in M$ there exists $\theta(x)\in ]0,1[$ such that
\begin{equation*}
 \begin{split}
 E_{\mu}(u+v) - E_{\mu}(u) & - DE_{\mu}(u)[v] = \langle P_{g_0}v,v\rangle -8\int_M f_{\mu}e^{4(u+\theta v)} v^2 \,dV_{g_0}\\
 & \leq  || v ||_{\Htwo}^2 
      + 8||f_{\mu}||_{\infty}\int_M e^{4(u+\theta v)} v^2 \,dV_{g_0}\; .
 \end{split}
\end{equation*}  
Applying twice H\"older's inequality and by Sobolev's embedding we obtain
\[
\begin{split}
\int_M e^{4(u+\theta v)} v^2 \,dV_{g_0} &\leq \left\{\int_M e^{8(u+\theta v)} \,dV_{g_0}\right\}^{1/2}
||v||_{L^4(M;g_0)}^2 \\
&\leq \left\{\int_M e^{16u}\,dV_{g_0} \cdot \int_M e^{16\theta v}\,dV_{g_0} \right\}^{1/4}||v||^2_{\Htwo}\,.
\end{split}
\]
Hence, we have
\[
\begin{split}
E_{\mu}(u+v)\leq & E_{\mu}(u) + DE_{\mu}(u)[v] +\\
&+||v||^2_{\Htwo}\left[1+C\left\{\int_M e^{16u} \,dV_{g_0}\right\}^{1/4}\left\{\int_M e^{16\theta v} \,dV_{g_0}\right\}^{1/4}\right]\,.
\end{split}
\]
In order to bound $\int_M e^{16\theta v} \,dV_{g_0}$, we proceed as
\[
\begin{split}
\int_M e^{16\theta v} \,dV_{g_0} &= \int_{M\cap[v\leq 0]} e^{16\theta v} \,dV_{g_0} +
\int_{M\cap[v> 0]} e^{16\theta v} \,dV_{g_0}\\
& \leq 1 + \int_{M\cap[v> 0]} e^{16v} \,dV_{g_0}\\
& \leq 1 + \int_{M} e^{16v} \,dV_{g_0} \\
&\leq C
\end{split}
\]
where in the last passage we have used again Adams' inequality. Our claim follows.
\end{proof}

We are now able to prove the analogue of Proposition 3.2 in \cite{Borer-Galimberti-Struwe}:
\begin{proposition}\label{prop: monotonicity trick} 
Suppose that the map $\Lambda \ni \mu \mapsto c_{\mu}$ is differentiable at some $\mu>\lambda$
(compare \eqref{eqn: la c mu e' non increasing}).
Then there exists a sequence $(p_n)_{n\in\N}$ in $P$ and a corresponding sequence of points
$u_n=p_n(t_n)\in \Htwo$, $n\in\N$, such that
\begin{equation}\label{eqn: monotonicity trick 1}
  E_{\mu}(u_n)\to c_{\mu},\; \max_{0\le t\le 1}E_{\mu}(p_n(t)) \to c_{\mu},\; ||DE_{\mu}(u_n)|| 
  \to 0 \hbox{ as } n \to \infty,\end{equation}
and with $(u_n)$ satisfying, in addition, the ``entropy bound''
\begin{equation}\label{eqn: monotonicity trick 2}
  \int_M e^{4u_n}\,dV_{g_0} = \big|\frac{d}{d\mu}E_{\mu}(u_n)\big| \le |c'_{\mu}|+3,
  \hbox{ uniformly in } n. 
\end{equation}
\end{proposition}
\begin{proof}
The argument is very close to the proof of Proposition 3.2 appearing in \cite{Borer-Galimberti-Struwe}. Therefore, in the following the reasoning will just be outlined and for the details we refer to
the afore-mentioned paper. Since $\lambda_0<\frac{1}{4}$, we can assume that for any $\mu\in\Lambda$ we
have $|\lambda-\mu|<1$. Let $\mu\in\Lambda$ be a point of differentiability of $c_\mu$ and $\mu_n\in\Lambda$
a sequence of numbers with $\mu_n\downarrow\mu$ as $n\to\infty$. We may find a sequence of path $p_n\in P$
and a sequence of $t_n\in [0,1]$ such that, setting $u=p_n(t_n)$, we obtain
\[
\max_{t\in[0,1]} E_{\mu}(p_n(t)) \leq c_\mu + (\mu_n - \mu),\; n\in \N\,,
\]
\[
 c_{\mu_n} - (\mu_n - \mu) \leq E_{\mu_n}(u) 
\] 
and
\[
0 \leq \frac{E_\mu(u)-E_{\mu_n}(u)}{\mu_n -\mu} 
  = \int_M e^{4u}  \,dV_{g_0} \le |c'_{\mu}| + 3
\] 
for all $n\in \N$ sufficiently large.

From that and via Jensen's inequality we can bound 
\[
 4 \int_M u \,dV_{g_0} \le \log \left(\int_M e^{4u} \,dV_{g_0}\right) 
  \le \log(|c'_{\mu}| + 3)= C(\mu)<\infty
\]
uniformly for $n$ sufficiently large, which leads to 
\[
 ||u||^{2}_{H^2} + \int_M e^{4u} \,dV_{g_0}\le C_1(\mu)\,.
\]  
   
We now assume by contradiction that there exists $\delta>0$ such that $||DE_\mu(u)||\geq 2\delta$ for all
$n$ sufficiently large and where $u=p_n(t_n)$. With the help of Lemma \ref{lemma: stime per E e DE} and
similarly as done in \cite{Borer-Galimberti-Struwe}, we can construct a suitable comparison path $\tilde{p}_n$ which contradicts the definition of $c_{\mu_n}$. That concludes the proof.
\end{proof}

With the help of the previous Proposition we obtain:

\begin{proposition}\label{prop: esistenza di una seconda soluzione}
Let $\mu$ be a point of differentiability for the function $c_\mu$. Then the functional $E_\mu$ admits
a critical point $u^\mu$ at the energy level $c_\mu$ and with volume 
$\int_M e^{4u^\mu}\,dV_{g_0} \leq |c'_{\mu}|+3$.
\end{proposition}
\begin{proof}
Let $\mu$ be a point of differentiability for the function $c_\mu$: Proposition \ref{prop: monotonicity trick} guarantees the existence of a sequence of paths $(p_n)_n\subset P$ and of a sequence of points
$u_n=p_n(t_n)\in\Htwo$ such that \eqref{eqn: monotonicity trick 1},\eqref{eqn: monotonicity trick 2} and
\[
||u_n||^{2}_{H^2} + \int_M e^{4u_n} \,dV_{g_0}\le C
\]
are true, where $C$ depends on $\mu$ but not on $n\in\N$. Therefore, up to subsequences, we can assume that, as $n\to\infty$, $u_n\rightharpoonup u^\mu$ weakly in $\Htwo$ and $u_n\to u^\mu$ strongly in $L^q(M,g_0)$ for
any $q\geq 1$. Furthermore, by the compactness of the map $\Htwo\ni w\to e^{4w}\in L^p(M,g_0)$, we can also
assume $e^{4u_n}\to e^{4u^\mu}$ in $L^p(M,g_0)$ for any $p\geq 1$. From this last fact, it follows
$\int_M e^{4u^\mu}\,dV_{g_0} \leq |c'_{\mu}|+3$. Moreover, with an error term $o(1)$ as $n\to\infty$, we can
write
\[
\begin{split}
o(1)&=\frac{1}{2}DE_\mu(u_n)[u_n-u^\mu]\\
    &=\langle P_{g_0}u_n,u_n-u^\mu\rangle +2Q_{g_0}\int_M (u_n-u^\mu)\, dV_{g_0}+\\
    & \;\;\;\;\;\;-2\int_M f_\mu e^{4u_n}(u_n-u^\mu)\,dV_{g_0}\\ 
    &=\langle P_{g_0}u_n-u^\mu,u_n-u^\mu\rangle +o(1)\,,
\end{split}
\]
viz $u_n\to u^\mu$ strongly in $\Htwo$ as $n\to\infty$. Therefore, we deduce also
$E_\mu(u_n)\to E_\mu(u^\mu)=c_\mu$ and $DE_\mu(u_n)\to DE_\mu(u^\mu)$; thus $u^\mu$ is a critical point
for $E_\mu$.
\end{proof}

\subsection*{Proof of Proposition \ref{prop: stima logaritmica}}
Let $p_0\in M$ be such that $f_0(p_0)=0$ and $\lambda\in(0,\lambda_0]$. We define the smooth Riemannian metric $g=e^{2u_0}g_0$. We fix a natural number $N\geq 5$. Then we can find a smooth metric $\tilde{g}(N)$ conformal to $g$ such that
\begin{equation}\label{eqn: determinante della g tilde}
\det(\tilde{g}(N)) = 1 + O(r^N)\,, \;\;\; \mbox{as} \; r \downarrow 0\,,  
\end{equation}
where $r=|x|$ and $x$ are $\tilde{g}(N)$-normal coordinates at $p_0\simeq 0$ (see \cite{Lee-Parker87}). Since $p_0$ is an isolated point of maximum of $f_0$, for a suitable constant $L>0$ with $\sqrt{\lambda_0}<L$ we have that in these normal coordinates
\begin{equation}\label{eqn: espansione di f}
f_0(x)=\frac{1}{2}D^2f_0(p_0)\left[x,x\right]+O(|x|^3)\ge -\lambda/2 \hbox{ on } B_{\sqrt{\lambda}/L}(0)
\end{equation}
and $f_{\lambda}\ge \lambda/2$ on $B_{\sqrt{\lambda}/L}(0)$ for all $\lambda\in(0,\lambda_0]$.

Fix a cut-off function $\tau\in C^{\infty}_c(B_1(0))$ with $0\leq\tau\leq 1$ and
\[
\tau(t)=
\left\{\begin{array}{ll}
1, & \; |t|<1/2\,, \\
0, & \; |t|\geq 1\,.   
\end{array}
\right.
\]

Let $A_0>1$ and let $\xi\in C^{\infty}\left([0,\infty)\right) $ be defined as
\[
\xi(t)=
\left\{\begin{array}{ll}
t, & \; t\in[0,1]\,, \\
2, & \; t\geq 2\,,\\
\in[1,2], & \; t\in(1,2)\,,
\end{array}
\right.  
\]
with $\xi'\geq 0$ and 
\begin{equation}\label{eqn: stime per chi 1}
\sup_{t\geq 0}\xi'(t)\leq A_0\,.  
\end{equation}
For $\delta>0$ define 
\begin{equation}\label{eqn: definizione di chi delta}
\xi_\delta(t)= \delta\,\xi(t/\delta) \,. 
\end{equation}
We note that pointwise $\lim_{\delta\to +\infty}\xi_\delta(t)=t$, and the convergence is uniform on compact subsets. Furthermore, for any $t\geq 0$, it holds
\begin{equation}\label{eqn: stime per chi delta}
|\xi_\delta'(t)|=|\xi'(t/\delta)|\leq A_0\,, \;\; |\xi_\delta''(t)|=|\xi''(t/\delta)\delta^{-1}|\leq \delta^{-1}||\xi''||_{\infty} \,, 
\end{equation}
whereas obviously $||\xi''||_{\infty}:=\sup_{t\geq 0}|\xi''(t)|$.

We set $\delta=\delta(\lambda):=\frac12\log(1/\lambda)$ and define
\[
z_\lambda(x)=
\left\{\begin{array}{ll}
\xi_{\delta}\left(\log\left(\frac{1}{|x|}\right)\right) \tau(|x|), & \; \lambda\leq |x|\leq 1\,, \\
\log(1/\lambda), &\; |x|\leq\lambda\,, \\  
0, &\; |x|>1\,.
\end{array}
\right.
\]
Then $z_\lambda\in C^\infty_c(\R^4)$ with $supp\, z_\lambda\subset\overline{B_1(0)}$. Finally, we define for $x\in B_{\frac{\sqrt{\lambda}}{L}}(0)$
\begin{equation}\label{eqn: w lambda}
w_\lambda(x)=z_\lambda\left(\frac{Lx}{\sqrt{\lambda}}\right)  
\end{equation}
and we extend $w_\lambda=0$ outside $B_{\frac{\sqrt{\lambda}}{L}}(0)$. Therefore, $w_\lambda\in C^\infty(M)$ with
$supp\,w_\lambda\subset\overline{B_{\frac{\sqrt{\lambda}}{L}}(0)}$. The euclidean gradient and Laplacian of $z_\lambda$ are respectively
\begin{equation}\label{eqn: gradiente euclideo di w lambda}
\nabla_{\R^4}z_\lambda(x)=  
\end{equation}
\[
=\left\{\begin{array}{ll}
0\,, &\;\; \mbox{if  }|x|\leq  \lambda\,,\\
&\\
-\xi'\left(\delta^{-1}\log\left(\frac{1}{|x|}\right)  \right) \frac{x}{|x|^2}\,, &\;\; \mbox{if  } \lambda\leq |x|\leq \sqrt{\lambda}\,,  \\
&\\
-\frac{x}{|x|^2}\,, &\;\; \mbox{if  }\sqrt{\lambda}\leq |x|\leq \frac{1}{2}\,,  \\
&\\
-\frac{x}{|x|^2}\tau\left(|x|\right) +
\log\left(\frac{1}{|x|}  \right)\tau'\left(|x|\right)\frac{x}{|x|}\,, &\;\; \mbox{if  }\frac{1}{2}\leq |x|\leq 1\,,  
\end{array}
\right.
\]
and
\begin{equation}\label{eqn: laplaciano euclideo di w lambda}
\Delta_{\R^4}z_\lambda(x)=
\end{equation}
\[
=\left\{\begin{array}{ll}
0\,, &\; \mbox{if  } |x|\leq  \lambda\,,\\
&\\
|x|^{-2}\left[\delta^{-1}\xi''\left(\delta^{-1}\log\left(\frac{1}{|x|}\right)  \right) -
2\xi'\left(\delta^{-1}\log\left(\frac{1}{|x|}\right) \right)\right] &\; \mbox{if  }\lambda\leq |x|\leq \sqrt{\lambda}\,,  \\
&\\
-2|x|^{-2}\,,   &\; \mbox{if  } \sqrt{\lambda}\leq |x|\leq \frac{1}{2}\,,  \\
&\\
-\tau'\left(|x|\right)|x|^{-1}\left[ 
2 +5\log\left(\frac{1}{|x|} \right) \right]+ & \\
&\\
\;\;\;\;\;-2\tau\left(|x|\right)|x|^{-2}
+ \log\left(\frac{1}{|x|} \right) \tau''\left(|x|\right)\,, &\; \mbox{if  }\frac{1}{2}\leq |x|\leq 1 \,.
\end{array}  
\right.
\]
\begin{lemma}\label{lemma: I lemma strumentale}
For any $0<\varepsilon <1$ there exist $\lambda_{\varepsilon}\in (0,\lambda_0)$, $C=C(g_0,f_0)>0$ and $C_N>0$ such that for any $0<\lambda<\lambda_{\varepsilon}$ and for any $s>0$ we have
\[
- \int_M f_\lambda e^{4(u_0+sw_\lambda)}dV_{g_0} \leq C -C_N(1-\varepsilon ) \lambda^{8-4s}
\]
uniformly in $A_0>1$.
\end{lemma}
\begin{proof}
Let $s>0$ and let $\varphi_N\in C^\infty(M)$ be the conformal factor $\tilde{g}(N)=e^{2\varphi_N}g=e^{2\varphi_N+2u_0}g_0$. Recalling that $w_\lambda$ is supported in $\overline{B_{\frac{\sqrt{\lambda}}{L}}(0)}$ and
equation (\ref{eqn: espansione di f}), we obtain
\[
\begin{split}
\int_M f_\lambda e^{4(u_0+sw_\lambda)}dV_{g_0} &= \int_M f_\lambda e^{4(sw_\lambda-\varphi_N)}dV_{\tilde{g}(N)} \\
&\geq \frac{\lambda}{2} \int_{B_{\frac{\sqrt{\lambda}}{L}}(0)} e^{4(sw_\lambda-\varphi_N)}dV_{\tilde{g}(N)}
-||f_0||_{\infty}\int_M e^{-4\varphi_N}dV_{\tilde{g}(N)}  \\
&= \frac{\lambda}{2} \int_{B_{\frac{\sqrt{\lambda}}{L}}(0)} e^{4(sw_\lambda-\varphi_N)}dV_{\tilde{g}(N)}
-\underbrace{||f_0||_{\infty}\int_M e^{4u_0}dV_{g_0}}_{=:C}\,.
\end{split}
\]
From (\ref{eqn: determinante della g tilde}) we have $dV_{\tilde{g}(N)}=\sqrt{1 + O(r^N)}dx$. Thus, given $0<\varepsilon <1$, there exists $\lambda_{\varepsilon}\in (0,\lambda_0)$, independent of $s>0$, such that for any $0<\lambda<\lambda_{\varepsilon}$
\[
\begin{split}
\int_{B_{\frac{\sqrt{\lambda}}{L}}(0)} e^{4(sw_\lambda-\varphi_N)}dV_{\tilde{g}(N)} & \geq
\min_{M}e^{-4\varphi_N}\int_{B_{\frac{\sqrt{\lambda}}{L}}(0)} e^{4sw_\lambda}\sqrt{1 + O(r^N)}\,dx  \\
& \geq \min_{M}e^{-4\varphi_N}(1-\varepsilon )\int_{B_{\frac{\sqrt{\lambda}}{L}}(0)} e^{4sw_\lambda}dx\\
& = C_N\frac{4L^4}{\pi^2}(1-\varepsilon )\int_{B_{\frac{\sqrt{\lambda}}{L}}(0)} e^{4sw_\lambda}dx\,,
\end{split}
\]
where $C_N:= \frac{\pi^2}{4L^4}\min_{M}e^{-4\varphi_N}$.
Recalling (\ref{eqn: w lambda}) and the definition of $z_\lambda$,
\[
\lambda \int_{B_{\frac{\sqrt{\lambda}}{L}}(0)} e^{4sw_\lambda}dx =
\frac{\lambda^3}{L^4}\int_{B_1(0)} e^{4sz_\lambda(y)}dy  \geq 
\frac{\lambda^{3-4s}}{L^4}\int_{B_{\lambda^{5/4}}(0)} dy = \frac{\pi^2\lambda^{8-4s}}{2L^4} 
\]
and therefore we conclude
\[
\int_M f_\lambda e^{4(u_0+sw_\lambda)}dV_{g_0} \geq C_N(1-\varepsilon ) \lambda^{8-4s} -C\,.
\]
\end{proof}
We note that in the conformal normal coordinates $\left\{x^i\right\}$ associated to $\tilde{g}(N)$, one has
for a radial function $v$ the following expansion
\begin{equation}\label{eqn: laplaciano di una radiale}
\Delta_{\tilde{g}(N)}v = \Delta_{\R^4}v + O''(r^{N-1})v'\,,  
\end{equation}
where $h\in O''(r^{N-1})$ if and only if $|\nabla^jh(x)|\leq C_j r^{N-1-j}$ for some constant $C_j$, $j=1,2$, and where $r=|x|=d_{\tilde{g}(N)}(x,p_0)$ (for a proof of that see for instance \cite{Gursky-Malchiodi15}).
Furthermore, if $\hat{g}(N)$ indicates the metric $\tilde{g}(N)$ written in polar coordinates $(r,\underline{\theta})$, one has $\sqrt{|\hat{g}(N)|}=r^3\sqrt{|\tilde{g}(N)|}$ and
\begin{equation}\label{eqn: norma del gradiente di una radiale}
\left|\nabla_{\hat{g}(N)}v\right|_{\hat{g}(N)}^2= \hat{g}^{rr}(v')^2\,.  
\end{equation}

In view of (\ref{eqn: laplaciano di una radiale}), which considerably simplifies the expression of the Laplacian and exploiting the conformal invariance of the Paneitz operator, we are able to show
\begin{lemma}\label{lemma: II lemma strumentale}
Given $0<\varepsilon<1$ and $A_0>1$, there exists $\lambda^{\varepsilon}\in(0,\lambda_0)$ independent of $A_0$ such that for all
$0<\lambda<\lambda^{\varepsilon}$ 
\[
\langle P_{g_0}w_\lambda,w_\lambda\rangle\leq 4\pi^2(1+\varepsilon)(A_0^2+1)\log\left(1/\lambda\right) + C_0\,,   
\] 
where $C_0$ depends at most quadratically on the supremum norm of $\xi''$ but it does not depend neither on $\lambda$ nor on $\varepsilon$.
\end{lemma}
\begin{proof}
Since the Paneitz operator is conformal invariant, we have
\[
\langle P_{g_0}w_\lambda,w_\lambda\rangle = \langle P_{\tilde{g}(N)}w_\lambda,w_\lambda\rangle\,, 
\]
where
\begin{equation}\label{eqn: paneitz in metrica normale conforme}
\begin{split}
\langle P_{\tilde{g}(N)}w_\lambda,w_\lambda\rangle &=
\int_M \left[\left( \Delta_{\tilde{g}(N)}w_\lambda\right)^2 +\frac{2}{3}R_{\tilde{g}(N)} 
\left|\nabla_{\tilde{g}(N)}w_\lambda\right|_{\tilde{g}(N)}^2 +\right.\\
&\\
& 
 -2\mbox{Ric}_{\tilde{g}(N)}\left(\nabla_{\tilde{g}(N)}w_\lambda,\nabla_{\tilde{g}(N)}w_\lambda\right) \bigg] dV_{\tilde{g}(N)}\,.
\end{split}
\end{equation}
Let's estimate first the term involving the Laplacian: given $\varepsilon >0$, there exists $\lambda^{\varepsilon}\in (0,\lambda_0)$ such that for $0<\lambda<\lambda^{\varepsilon}$  
\begin{equation}\label{eqn: laplaciano in metrica normale conforme}
\begin{split}
\int_M \left( \Delta_{\tilde{g}(N)}w_\lambda\right)^2 dV_{\tilde{g}(N)} &=
\int_{B_{\frac{\sqrt{\lambda}}{L}}(0)} \left( \Delta_{\tilde{g}(N)}w_\lambda\right)^2 \sqrt{1 + O(r^N)}\,dx\\
& \leq
(1+\varepsilon)\int_{B_{\frac{\sqrt{\lambda}}{L}}(0)} \left( \Delta_{\tilde{g}(N)}w_\lambda\right)^2 dx\\
& =(1+\varepsilon)\int_{B_1(0)} \left( \Delta_{\tilde{g}(N)}z_\lambda\right)^2 dx
\end{split}
\end{equation}
and, from (\ref{eqn: laplaciano di una radiale}),
\[
\int_{B_1(0)} \left( \Delta_{\tilde{g}(N)}z_\lambda\right)^2 dx=
\]
\[
=\int_{B_1(0)}\left[ \left( \Delta_{\R^4}z_\lambda\right)^2 +
2\Delta_{\R^4}z_\lambda \, z_\lambda' O''(r^{N-1})+\left(z_\lambda' O''(r^{N-1})\right)^2\right]dx
\]
\[
=: M_1+M_2+M_3\,.
\]
We have (see Appendix \textbf{A}) for $0<\lambda<\lambda^{\varepsilon}$
\begin{equation}\label{eqn: M1}
M_1\leq 4\pi^2(A_0^2+1)\log\left(1/\lambda\right) + C_0 \,,
\end{equation}
where $C_0$ is a costant depending at most quadratically on the supremum norm of $\xi''$, independent of $\lambda$ and which is allowed to vary from line to line;
\begin{equation}\label{eqn: M2}
M_2=O(\lambda^{\frac{N-3}{2}})   
\end{equation}
and
\begin{equation}\label{eqn: M3}
M_3=O(\lambda^{N-3})
\end{equation}
as $\lambda\downarrow 0$. Hence, by choosing a smaller $\lambda^{\varepsilon}$, if necessary, and recalling that by assumption $\varepsilon<1$, for all $0<\lambda<\lambda^{\varepsilon}$ we obtain
\begin{equation}\label{eqn: stima logaritmica per il laplaciano}
\int_M \left( \Delta_{\tilde{g}(N)}w_\lambda\right)^2 dV_{\tilde{g}(N)}\leq
4\pi^2(1+\varepsilon)(A_0^2+1)\log\left(1/\lambda\right) + C_0\,,   
\end{equation}
where as above $C_0$ depends at most quadratically on the supremum norm of $\xi''$ but it does not depend neither on $\lambda$ nor on $\varepsilon$.

For the remaining part of the Paneitz operator we have
\[
\int_M \left[\frac{2}{3}R_{\tilde{g}(N)} \left|\nabla_{\tilde{g}(N)}w_\lambda\right|_{\tilde{g}(N)}^2 
-2\mbox{Ric}_{\tilde{g}(N)}\left(\nabla_{\tilde{g}(N)}w_\lambda,\nabla_{\tilde{g}(N)}w_\lambda\right) \right] dV_{\tilde{g}(N)}
\]
\[
\leq
\]
\[
C \int_M \left|\nabla_{\tilde{g}(N)}w_\lambda\right|_{\tilde{g}(N)}^2  dV_{\tilde{g}(N)} =
C \int_{B_{\frac{\sqrt{\lambda}}{L}}(0)} \left|\nabla_{\tilde{g}(N)}w_\lambda\right|_{\tilde{g}(N)}^2  dV_{\tilde{g}(N)}\,, 
\]
where $C=C(M,g_0,N)$. Therefore, from (\ref{eqn: norma del gradiente di una radiale}) for all $0<\lambda<\lambda^{\varepsilon}$ we have
\[
\begin{split}
\int_{B_{\frac{\sqrt{\lambda}}{L}}(0)} \left|\nabla_{\tilde{g}(N)}w_\lambda\right|_{\tilde{g}(N)}^2  dV_{\tilde{g}(N)}& \leq 2(1+\varepsilon) \int_{B_{\frac{\sqrt{\lambda}}{L}}(0)} (w_\lambda')^2 r^3drd\underline{\theta}\\
& =2(1+\varepsilon)\frac{\lambda}{L^2} \int_{B_{1(0)}} (z_\lambda')^2 r^3drd\underline{\theta}\,.
\end{split}
\]
In a way analogous to what has already been done in the Appendix and recalling \eqref{eqn: gradiente euclideo di w lambda}, we infer that $\int_{B_{\frac{\sqrt{\lambda}}{L}}(0)} \left|\nabla_{\tilde{g}(N)}w_\lambda\right|_{\tilde{g}(N)}^2  dV_{\tilde{g}(N)}=O(\lambda)$ as $\lambda\downarrow 0$. From that and from (\ref{eqn: stima logaritmica per il laplaciano}), we conclude
\[
\langle P_{g_0}w_\lambda,w_\lambda\rangle\leq 4\pi^2(1+\varepsilon)(A_0^2+1)\log\left(1/\lambda\right) + C_0\,,   
\] 
which holds for $0<\lambda<\lambda^{\varepsilon}$.
\end{proof}
Before terminating the proof of Proposition \ref{prop: stima logaritmica}, we observe that, since the Paneitz operator is assumed to be non-negative, it defines a semi-inner product on $\Htwo$ and hence the Cauchy-Schwartz inequality holds true
\[
|\langle P_{g_0}u_1,u_2\rangle| \leq \sqrt{\langle P_{g_0}u_1,u_1\rangle}\sqrt{\langle P_{g_0}u_2,u_2\rangle}
\,, \;\; u_1,u_2\in\Htwo
\]
and hence for any $t>0$ we have
\begin{equation}\label{eqn: Young per il paneitz}
|\langle P_{g_0}u_1,u_2\rangle| \leq t \langle P_{g_0}u_1,u_1\rangle +t^{-1}\langle P_{g_0}u_2,u_2\rangle \,.  
\end{equation}
\begin{proof}[Proof of Proposition \ref{prop: stima logaritmica} (completed)]
Given $K>32\pi^2$, we can find suitable numbers (not unique)
$0<\varepsilon<1$, $\alpha>0$ and $1<A_0<2$ such that 
\[
K>4\left[ 4\pi^2(1+\varepsilon)(A_0^2+1)+\alpha\right]\,.
\]
According to Lemma \ref{lemma: II lemma strumentale},
there exists $\lambda^{\varepsilon}\in(0,\lambda_0)$ such that for $0<\lambda<\lambda^{\varepsilon}$
\[
\langle P_{g_0}w_\lambda,w_\lambda\rangle\leq 4\pi^2(1+\varepsilon)(A_0^2+1)\log\left(1/\lambda\right) + C_0\,.  
\]
Furthermore, given our $\alpha>0$, it is possible to find $\lambda(\alpha,A_0)<\lambda^{\varepsilon}$ such that for $0<\lambda<\lambda(\alpha,A_0)$
\begin{equation}\label{eqn: stima paneitz con tutte le 3 costanti}
\langle P_{g_0}w_\lambda,w_\lambda\rangle\leq \left[4\pi^2(1+\varepsilon)(A_0^2+1)+\alpha\right]\log\left(1/\lambda\right)\,.  
\end{equation}
Define $\lambda_K:=\min\left\{\lambda_{\varepsilon},\lambda(\alpha,A_0),\lambda_0/2\right\}$, where $\lambda_\varepsilon$ is given by Lemma \ref{lemma: I lemma strumentale}, and consider $0<\lambda<\lambda_K$. Set
\[
\delta:=\frac{K -4\left[4\pi^2(1+\varepsilon)(A_0^2+1)+\alpha\right]}{8\left[4\pi^2(1+\varepsilon)(A_0^2+1)+\alpha\right]}\,, \; 
K_1:= \frac{K +4\left[4\pi^2(1+\varepsilon)(A_0^2+1)+\alpha\right]}{2}
\]
and note that $\delta>0$. Thus, by (\ref{eqn: Young per il paneitz}) and (\ref{eqn: stima paneitz con tutte le 3 costanti}), we can bound
\[
\begin{split}
\langle P_{g_0}u_0+sw_\lambda,u_0+sw_\lambda\rangle &\leq
\left(1+4/\delta\right)\langle P_{g_0}u_0,u_0\rangle+s^2(1+\delta)\langle P_{g_0}w_\lambda,w_\lambda\rangle \\
&\leq \left(1+4/\delta\right)\langle P_{g_0}u_0,u_0\rangle+K_1\frac{s^2}{4}\log\left(1/\lambda\right)\,.
\end{split}
\]
Because $w_\lambda\geq 0$ and $Q_{g_0}<0$, for every $s>0$ we have
\[
Q_{g_0}\int_M (u_0+sw_\lambda) \, dV_{g_0}\leq Q_{g_0}\int_M u_0\,  dV_{g_0}\,;
\]
therefore, with a constant $\overline{C}=\overline{C}(u_0,f_0,K)$, we obtain, in view of Lemma \ref{lemma: I lemma strumentale}, that for any $s>0$ and any $0<\lambda<\lambda_K$
\[
E_\lambda(u_0+sw_\lambda) \leq K_1\frac{s^2}{4}\log\left(1/\lambda\right)-C_N(1-\varepsilon)\lambda^{8-4s}+\overline{C}\,,
\]
where $C_N$ depends only on the fixed $N$. From this, we see that, for any fixed $0<\lambda<\lambda_K$,
$E_\lambda(u_0+sw_\lambda)\to-\infty$ as $s\to\infty$ and therefore we may fix some $s_\lambda>2$ with
$v_\lambda=u_0+s_\lambda w_\lambda$ satisfying $E_\lambda(v_\lambda)<\beta_0$ to obtain
\[
c_\lambda \leq \sup_{s>0} E_\lambda(u_0+sw_\lambda)\leq \sup_{s>0}\left[K_1\frac{s^2}{4}\log\left(1/\lambda\right)-C_N(1-\varepsilon)\lambda^{8-4s}+\overline{C}\right]\,.
\]
For any $0<\lambda<\lambda_K$ the supremum in the latter quantity is achieved for some $s=s(\lambda)>2$, with
$s=s(\lambda)\to 2$ as $\lambda\downarrow 0$. Thus, taking a smaller $\lambda_K$ if necessary, we obtain eventually
\[
c_\lambda \leq K \log\left(1/\lambda\right)\,.
\]
Furthermore, since $E_\mu(v_\lambda)\leq E_\lambda(v_\lambda)$ for $\mu>\lambda$, the same comparison function
$v_\lambda$ can be used for every $\mu\in\Lambda:=(\lambda,2\lambda)\subset\Lambda_0$, and for these $\mu$
we obtain the estimate
\[
E_\mu(v_{\lambda}) < E_\mu(u_\mu) \le \sup_{\nu \in \Lambda} E_\mu(u_\nu)
  < \beta_0\le c_\mu\le K\log(1/\lambda)\le K\log(2/\mu),
\]
where $\beta_0$ and $c_\mu$ for $\mu \in\Lambda$ are as defined in \eqref{eqn: struttura mountain-pass} and \eqref{eqn: definizione di c_lambda}. The claim follows and Proposition \ref{prop: stima logaritmica} is proved.
\end{proof}
 
\section{Proof of Theorem \ref{thm: bubbling analysis}} 
 
Proposition \ref{prop: esistenza di una seconda soluzione} guarantees the existence of a sequence $(\lambda_n)_n$ such that $\lambda_n\downarrow 0$ as $n\to\infty$ and of a sequence $u_n:=u^{\lambda_n}$ of ``large'' solutions of \eqref{eqn: equazione target} with $f_{\lambda_n}$. Now in order to analyze the behaviour of the ``limit'' geometry of the manifolds $(M,e^{2u_n}g_0)$ when $\lambda_n\downarrow 0$ and to prove that $u_n$ blows up in a spherical bubble, one would like to resort to the results of \cite{Martinazzi2009} or \cite{Malchiodi2006} for instance. However, similarly to the situation occuring in the two dimensional case (\cite{Borer-Galimberti-Struwe}), the afore-mentioned results require either a uniform bound on the volume of the manifolds $(M,e^{2u_n}g_0)$ or that the function $f_{\lambda_n}$ does not change the sign, assumptions which clearly do not hold in the present case. In order to overcome these obstacles, we will resort to the ``entropy'' bound given by Proposition \ref{prop: esistenza di una seconda soluzione}. 

Reasoning as in \cite{Borer-Galimberti-Struwe}, we obtain the following result:
\begin{lemma}\label{lemma: mu e c primo mu}
We have $\liminf_{\mu\downarrow 0}(\mu|c'_{\mu}|)\leq 32\pi^2.$
\end{lemma}
\begin{proof}
Otherwise there are two constants $K>K_1>32\pi^2$ and $\mu_0>0$ such that $\inf\left\{\mu|c'_{\mu}|:
0<\mu\leq\mu_0\,, \exists c'_{\mu} \right\}>K$. Hence, by Lebesgue Theorem for every $0<\mu_1<\mu_0$ we
have 
\[
c_{\mu_1}\geq c_{\mu_0}+\int_{\mu_1}^{\mu_0}|c'_\mu|d\mu \geq c_{\mu_0}+K\log\left(\mu_0/\mu_1\right)\,.
\]
On the other hand, by means of Proposition \ref{prop: stima logaritmica} we have for all sufficiently small $\mu_1>0$ that $c_{\mu_1}\leq K_1\log\left(\mu_0/\mu_1\right)$, which contradicts the above inequality.
\end{proof}
Now observe that by Propostion \ref{prop: esistenza di una seconda soluzione} for almost every sufficiently small
$\mu>0$ the second solution which we have obtained satisfies the volume bound $\int_M e^{4u^\mu}\,dV_{g_0} \leq |c'_{\mu}|+3$. After replacing $\mu$ with $\lambda$, we then have a sequence of ``large'' solutions 
$u_n:=u^{\lambda_n}$ of \eqref{eqn: equazione target} for $f_{\lambda_n}$ and with $\lambda_n\downarrow 0$
satisfying
\begin{equation}\label{eqn: volume bound}
\limsup_n\left(\lambda_n \int_M e^{4u_n}\,dV_{g_0} \right)\leq 32\pi^2\,.  
\end{equation}
Equation \eqref{eqn: k_P} now reads for the metric $e^{2u_n}g_0$ as
\[
k_P=\int_M f_0e^{4u_n}\,dV_{g_0}+\lambda_n\int_M e^{4u_n}\,dV_{g_0}\,,
\]
which in view of \eqref{eqn: volume bound} leads to the global $L^1$-bound
\begin{equation}\label{eqn: global L1 bound}
\sup_n \int_M (|f_0|+\lambda_n)e^{4u_n}\,dV_{g_0} <\infty\,.
\end{equation} 

Since $u_n$ is at least $C^4$, we have the following representation formula
\[
u_n(x)=\bar{u}_n+\int_M G(x,y)P_{g_0}u_n \,dV_{g_0}(y)\,, \;\;\; x\in M\,,
\]
where $G$ is the Green function for $P_{g_0}$ (compare Lemma 1.7 \cite{Chang-Yang95}). We set
\[
\gamma_n:=2f_{\lambda_n}e^{4u_n}-2Q_{g_0}
\]
and observe that for any $n\in\N$ the quantity $||\gamma_n||_{L^1(M)}\neq 0$, otherwise $P_{g_0}u_n=0$, $u_n=const.$ and hence $f_{\lambda_n}=const.$ Therefore, reasoning as in Lemma 2.3 \cite{Malchiodi2006}, one obtains for $j=1,2,3$
\[
|\nabla_{g_0}^ju_n|_{g_0}^p(x)\leq C(M,g_0) \int_M \left( \frac{||\gamma_n||_{L^1(M)}}{|x-y|^3} \right)^p
\,\frac{|\gamma_n(y)|}{||\gamma_n||_{L^1(M)}} \,dV_{g_0}(y)\,, 
\]
for a.e. $x\in M$. In view of the global $L^1$-bound given by \eqref{eqn: global L1 bound}, by means of Jensen's inequality and Fubini's theorem, and arguing as in \cite{Malchiodi2006}, we deduce the bound
\[
\sup_n\int_M \left(|\nabla_{g_0}^3u_n|^p_{g_0} +|\nabla_{g_0}^2u_n|^p_{g_0} +|\nabla_{g_0}u_n|^p_{g_0} \right)\, dV_{g_0} < \infty
\]
for any $p\in\left[1,4/3\right)$. By Poincar\'e's inequality we also have $\int_M |u_n-\bar{u}_n|^p\, dV_{g_0}\leq C$ uniformly in $n$; therefore, setting 
\begin{equation}\label{eqn: ausiliaria v_n}
v_n:=u_n-\bar{u}_n\,,  
\end{equation}
we deduce the boundness of the sequence $(v_n)_n$ in $W^{3,p}(M;g_0)$ for all $p\in\left[1,4/3\right)$. Therefore, by Sobolev embedding results we infer that:
\begin{enumerate}
\item $(v_n)_n$ is bounded in $L^q(M)$ for any $q\in[1,\infty);$    
\item $(\nabla_{g_0} v_n)_n$ is bounded in $L^r(M)$ for any $r\in[1,4);$ 
\item $(\nabla^2_{g_0} v_n)_n$ is bounded in $L^s(M)$ for any $s\in[1,2)$,   
\end{enumerate}
a result needed later. Observe also that $v_n$ solves
\begin{equation}\label{eqn: l'eq risolta da v_n}
P_{g_0}v_n + 2Q_{g_0} = 2f_{\lambda_n}e^{4v_n}e^{4\bar{u}_n}\;\;\;\; \mbox{on  } M\,.
\end{equation}

Noting that $||\,|f_0| +\lambda_n||_{L^1(M)}\geq ||f_0||_{L^1(M)}>0 $ uniformly in $n\in\mathbb{N}$, we define the $\left(|f_0|+\lambda_n\right)$-average of $u_n$ as
\[
\bar{\bar{u}}_n := \int_M u_n \left(|f_0|+\lambda_n\right)\, \frac{dV_{g_0}}{||\,|f_0| +\lambda_n||_{L^1(M)}}\,.
\]
Hence, in view of \eqref{eqn: global L1 bound} and Jensen's inequality, we infer the bound
\[
C\geq ||\,|f_0| +\lambda_n||_{L^1(M)} e^{4\bar{\bar{u}}_n}\geq ||f_0||_{L^1(M)}e^{4\bar{\bar{u}}_n}
\]
and consequently $\sup_n \bar{\bar{u}}_n < \infty$. Arguing as in Lemma 4.2 by \cite{Borer-Galimberti-Struwe}, we can show the existence of a positive constant $C$ independent of $n$ such that the
following Poincar\'e type inequality holds
\[
||u_n-\bar{\bar{u}}_n||_{L^2(M)}\leq C ||\nabla_{g_0}u_n||_{L^2(M)}\,.
\]
Therefore, thanks to that and to the ``classical'' Poincar\'e's inequality, we obtain that, uniformly in $n$, 
\[
|\bar{u}_n -\bar{\bar{u}}_n|\leq C ||\nabla_{g_0}u_n||_{L^2(M)} = C ||\nabla_{g_0}v_n||_{L^2(M)}\,,
\]
and we conclude by above that $\sup_n |\bar{u}_n -\bar{\bar{u}}_n|< \infty$. Since we know that $\bar{\bar{u}}_n\leq C$ uniformly in $n$, we finally obtain for our sequence of solutions $(u_n)_n$ that
\begin{equation}\label{eqn: la media e' limitata da sopra, almeno quello...}
\sup_n \bar{u}_n <\infty\,.  
\end{equation}
\begin{lemma}\label{lemma: pre-compattezza}
Let $(v_n)_n$ be the sequence defined by \eqref{eqn: ausiliaria v_n}. Then for any domain $\Omega\subset\subset M^-:= \left\{p\in M: f_0(p)<0\right\}$ we have
\[
\sup_n \int_\Omega \left(\Delta_{g_0}v_n\right)^2\,dV_{g_0}\leq C(\Omega)\,.
\]
\end{lemma}
\begin{proof}

Given any domain $\Omega\subset\subset M^-$, we notice that it is enough to prove the result on an arbitrary metric ball $B_{d}(p)\subset\subset M^-$, since afterward the estimate for $\Omega$ can be deduced by a covering argument. Thus, let $B_{4d}=B_{4d}(p)$ be such a ball, where $d$ is chosen small enough to guarantee that we stay in a single chart. Let $0\leq \eta\leq 1$ be a smooth cut-off function whose support is $\overline{B_{2d}}$ and $\eta=1$ on $B_d$. Therefore, $\eta^2v_n\in C^{\infty}(M)$ and it is supported in $B_{2d}$. In the following, to alleviate our notation, we set for any $\alpha,\beta\in\Htwo$
\[
D\left(\nabla_{g_0}\alpha,\nabla_{g_0}\beta\right):=
\frac{2}{3}R_{g_0} g_0\left(\nabla_{g_0}\alpha,\nabla_{g_0}\beta\right) -2\mbox{Ric}_{g_0}\left(\nabla_{g_0}\alpha,\nabla_{g_0}\beta\right)\,.
\]
A straightforward computation shows that
\begin{equation}\label{eqn: precompattezza 1}
\langle P_{g_0}v_n,\eta^2v_n\rangle-\langle P_{g_0}\eta v_n,\eta v_n\rangle=
\end{equation}
\[
\begin{split}
&= 
 \int_M \left[ \eta v_n\Delta_{g_0}v_n\Delta_{g_0}\eta +2\Delta_{g_0}v_ng_0\left(\nabla_{g_0}\eta,\nabla_{g_0}(\eta v_n)\right)\right] dV_{g_0}+\\
&-\int_M \left[
v_n\Delta_{g_0}\eta\Delta_{g_0}(\eta v_n)-2\Delta_{g_0}(\eta v_n) g_0\left(\nabla_{g_0}v_n,\nabla_{g_0}\eta\right)\right]dV_{g_0} +\\
&+\int_M\left[D\left(\nabla_{g_0}v_n,\eta v_n\nabla_{g_0}\eta\right) -
D\left(v_n\nabla_{g_0}\eta,\nabla_{g_0}(\eta v_n)\right)\right] dV_{g_0}= \\
&=\int_M \left[2v_n\Delta_{g_0}v_n|\nabla_{g_0}\eta|^2_{g_0}-\left(v_n\Delta_{g_0}\eta\right)^2-
v_n^2 D\left(\nabla_{g_0}\eta,\nabla_{g_0}\eta\right)\right] dV_{g_0}+\\
&-4\int_M \left[v_n\Delta_{g_0}\eta\, g_0\left(\nabla_{g_0}v_n,\nabla_{g_0}\eta\right)+            
g_0\left(\nabla_{g_0}v_n,\nabla_{g_0}\eta\right)^2 \right] dV_{g_0}\,.
\end{split}  
\]

In view of the bounds for $v_n$ and its derivative deduced above (compare immediately after \eqref{eqn: ausiliaria v_n}) and by several applications of H\"older's inequality (see Appendix \textbf{B}), we deduce that there exists a constant $C(\eta,d)$ such that
\begin{equation}\label{eqn: precompattezza 2}
\left| \langle P_{g_0}v_n,\eta^2v_n\rangle-\langle P_{g_0}\eta v_n,\eta v_n\rangle\right|\leq C(\eta,d)
\end{equation}
uniformly in $n$. With the same reasoning we can bound for all $n$ the quantity
\[
\left|2Q_{g_0}\int_M \eta^2v_n dV_{g_0} \right|  \leq C(\eta,d)\,.
\]
Hence, integrating by parts \eqref{eqn: l'eq risolta da v_n} with $\eta^2v_n$, we infer
\[
0\leq\langle P_{g_0}\eta v_n,\eta v_n\rangle\leq C(\eta,d) +2e^{4\bar{u}_n}\int_M f_{\lambda_n}e^{4v_n}\eta^2v_n\,dV_{g_0}\,.
\]
In view of \eqref{eqn: poincare per il laplaciano} it thus follows
\[
\begin{split}
\int_{B_d}\left(\Delta_{g_0}v_n  \right)^2\,dV_{g_0} &=\int_{B_d}\left(\Delta_{g_0}(\eta v_n) \right)^2\,dV_{g_0} \leq \int_M\left(\Delta_{g_0}(\eta v_n)  \right)^2\,dV_{g_0} \\
&\leq C\left(C(\eta,d) +2e^{4\bar{u}_n}\int_M f_{\lambda_n}e^{4v_n}\eta^2v_n\,dV_{g_0}   \right)
\end{split}
\]
and hence our claim will follow if we can bound from above uniformly in $n$ the last term on the right hand side. There exists $\epsilon>0$ such that, for all $n$ sufficiently large, $f_{\lambda_n}<-\epsilon$ on the ball $B_{2d}$.
Therefore, letting $B_n^+:=B_{2d}\cap\left[v_n>0\right]$ and $B_n^-:=B_{2d}\cap\left[v_n\leq 0\right]$, 
we obtain
\[
\begin{split}
&\int_M f_{\lambda_n}e^{4v_n}\eta^2v_n\,dV_{g_0}=\int_{B_{2d}} f_{\lambda_n}e^{4v_n}\eta^2v_n\,dV_{g_0}=\\
&=\int_{B_n^+} f_{\lambda_n}e^{4v_n}\eta^2v_n\,dV_{g_0}+
\int_{B_n^-} f_{\lambda_n}e^{4v_n}\eta^2v_n\,dV_{g_0}\leq \\
& \leq -\epsilon\int_{B_n^+} e^{4v_n}\eta^2v_n\,dV_{g_0}+\int_{B_n^-} f_{\lambda_n}\eta^2v_n\,dV_{g_0}\leq\\
& \leq 0+ ||\eta^2 f_{\lambda_n}||_{\infty} ||v_n||_{L^1}\,,
\end{split}
\]
which, as we have already seen, is uniformly bounded. Recalling now \eqref{eqn: la media e' limitata da sopra, almeno quello...}, the claim follows and the Lemma is proved.
\end{proof}

By the Lemma above and by reflexivity of the space $W^{3,p}(M;g_0)$, $p\in\left(1,4/3\right)$, we infer the existence of a subsequence still denoted $(v_n)_n$ such that, as $n\to\infty$, $v_n\rightharpoonup v_\infty$
in $W^{3,p}(M;g_0)$, $p\in\left(1,4/3\right)$ and
\begin{enumerate}
\item $v_n\to v_\infty$  in $L^q(M)$ for any $q\in[1,\infty);$    
\item $\partial_\alpha v_n\to \partial_\alpha v_\infty$ in $L^r(M)$ for any $r\in[1,4)$ and $|\alpha|=1;$ 
\item $\partial^2_\alpha v_n\to \partial^2_\alpha v_\infty$ in $L^s(M)$ for any $s\in[1,2)$ and $|\alpha|=2$. 
\end{enumerate}
Furthermore, we obtain for any domain $\Omega\subset\subset M^-$ 
\[
\sup_n ||v_n||_{H^2(\Omega)} \leq C(\Omega)\,.
\]

Given such a domain $\Omega$, we take a point $p\in\Omega$ and for a sufficiently small $\delta>0$ we consider the exponential map
\[
\exp_p: B_{4\delta}(0)\subset \R^4\to M\,; \;\;\;\;\; \exp_p(0)=p
\]
and the pull-back metric $\tilde{g}:=\left(\exp_p\right)^*g_0$ on $B_{4\delta}(0)$. Letting 
$\tilde{v}_n:=v_n\circ\exp_p$ and $\tilde{f_0}:=f_0\circ\exp_p$, we obtain by definition that $\tilde{v}_n$ solves the equation
\[
P_{\tilde{g}}\tilde{v}_n(x) + 2Q_{g_0} = 2e^{4\bar{u}_n}\tilde{f}_{\lambda_n}(x)e^{4\tilde{v}_n(x)}\,, 
\;\;\;\; x\in B_{4\delta}(0)\,.
\]
We consider $\chi\in C^\infty_{c}(B_{4\delta}(0)),0\leq\chi\leq 1$ and $\chi=1$ on $B_{2\delta}(0)$. Then $\chi v_n\in C^2_{c}(B_{4\delta}(0))$ and from above we infer $\sup_n ||\chi v_n||_{H^2(\Omega)} \leq C(\Omega)$. Therefore, in view of Adams' inequality and of \eqref{eqn: la media e' limitata da sopra, almeno quello...}, the sequence $(e^{4\bar{u}_n}\tilde{f}_{\lambda_n}e^{4\tilde{v}_n}-Q_{g_0})_n$ is bounded in $L^p(B_{2\delta}(0))$ for any $p\geq 1$; therefore, by standard elliptic regularity theory (see for instance Thm 7.1 \cite{Agmon59}), we conclude
\[
||\tilde{v}_n||_{W^{4,p}(B_\delta(0))} \leq C(\delta)\left(||e^{4\bar{u}_n}\tilde{f}_{\lambda_n}e^{4\tilde{v}_n}-Q_{g_0}||_{L^p(B_{\frac{3\delta}{2}}(0))}+ ||\tilde{v}_n||_{L^p(B_{\frac{3\delta}{2}}(0))}\right)\leq C(\delta)
\]
and hence, up to subsequences, that for any $p\geq 1$, as $n\to\infty$, $\tilde{v}_n\rightharpoonup \tilde{v}_\infty$ in $W^{4,p}(B_\delta(0))$, where we have set $\tilde{v}_\infty:=v_\infty\circ\exp_p$. By Sobolev embedding we obtain $\tilde{v}_n\to\tilde{v}_\infty$ strongly in $C^{2,\alpha}(B_\delta(0))$ with $\alpha\in[0,1)$ and eventually, by a covering argument, that
\begin{equation}\label{eqn: convergenza locale in C 2 alfa}
v_n\to v_\infty \;\;\; \mbox{in   }   C^{2,\alpha}(\Omega), \; \alpha\in [0,1)
\end{equation} 
as $n\to\infty$. 

We call a point $p\in M$ a blow-up point for the sequence $(u_n)_n$ if for any $r>0$ we have $\sup_{B_r(p)}u_n\to\infty$ as $n\to\infty$. We note that there must exist at least one blow-up point for our sequence of solutions $(u_n)_n$, since otherwise by regularity arguments we could extract a subsequence converging smoothly to the absolut minimizer of $E_{f_0}$. On the other hand, at this stage the structure of the blow-up set is not so clear and in principle one could expect it to have a ``rough'' shape (compare for instance \cite{Adimurthi-Robert-Struwe}).

The next result, which is essentially based on the the concentration-compactness criterion appearing in \cite{Malchiodi2006} Prop. 3.1., actually gives a precise description of the blow-up set.

\begin{lemma}\label{lemma: concentrazione energia}
Up to subsequences, we have that the blow-up set for the solutions $(u_n)_n$ satisfying \eqref{eqn: global L1 bound} is finite. Let $\left\{p_\infty^{(1)},\dots ,p_\infty^{(I)} \right\} $ be such blow-up points. Then, for any $1\leq i\leq I$, we have $f_0(p_\infty^{(i)})=0$ and for any $r>0$ there holds
\begin{equation}\label{eqn: concentrazione energia}
\liminf_n \int_{B_r(p_\infty^{(i)})} \left|f_{\lambda_n}\right|e^{4u_n} dV_{g_0}\geq 4\pi^2\,.
\end{equation}
\end{lemma}
\begin{proof}
Let $p\in M$ be a blow-up point for the sequence $(u_n)_n$. We assume that there exists $r_p>0$ such that for all $r<6r_p$ there holds
\begin{equation}\label{eqn: concentrazione energia, contraddizione}
\liminf_n \int_{B_r(p)} 2\left|f_{\lambda_n}e^{4u_n}-Q_{g_0} \right|dV_{g_0}< 8\pi^2\,.  
\end{equation}
We note that \eqref{eqn: global L1 bound} enables us to repeat the same reasoning appearing in \cite{Malchiodi2006} Prop. 3.1. locally on the ball $B_{3r_p}(p)$ (if one looks carefully at this proof, he will see that the arguments therein are local in nature). Therefore, we deduce the existence of a $\beta>1$ such that, up to subsequences,
\[
\sup_n\int_{B_{3r_p}(p)} e^{4\beta v_n}\,dV_{g_0}= \sup_n\int_{B_{3r_p}(p)} e^{4\beta(u_n-\bar{u}_n)}\,dV_{g_0}<\infty \,.
\]
Since $v_n$ solves \eqref{eqn: l'eq risolta da v_n}, taking account of \eqref{eqn: la media e' limitata da sopra, almeno quello...}, we see that the right hand side of \eqref{eqn: l'eq risolta da v_n} is bounded uniformly in $n$ in $L^{\beta}(B_{3r_p}(p))$ for some $\beta>1$. By standard elliptic regularity theory, similarly to what has been done above, and recalling that $(v_n)_n$ was bounded in any $L^q(M)$, we infer $||v_n||_{W^{4,\beta}(B_{2r_p}(p);g_0)}\leq C$ uniformly in $n$ and therefore, by Sobolev embedding, we conclude that the sequence $(v_n)_n$ is bounded at least in $C^{0,\alpha}(B_{2r_p}(p))$ for $\alpha\in [0,4-4/\beta]$. Thus, setting $u_n(x_n)=\sup_{B_{r_p}(p)}u$, observing that $u_n(x_n)\to\infty$, as $n\to\infty$, and recalling \eqref{eqn: la media e' limitata da sopra, almeno quello...}, we obtain
\[
u_n(x_n)\leq |u_n(x_n)-\bar{u}_n| +\bar{u}_n \leq |v_n(x_n)| + C \leq C
\]
uniformly in $n$, which is clearly a contradiction. Therefore, \eqref{eqn: concentrazione energia, contraddizione} cannot be true. From this, we deduce immediately \eqref{eqn: concentrazione energia} and, again from \eqref{eqn: global L1 bound}, we infer that the blow-up points are finite.

It remains to prove that they are all points of maximum of $f_0$. We assume that this is not the case and so that there exists a blow-up point $p\in M^-$. We now consider a small ball $B_r(p)\subset\subset M^-$ and infer, in view of \eqref{eqn: convergenza locale in C 2 alfa} that $v_n\to v_\infty$ in
$C^{2,\alpha}(B_r(p))$ as $n\to \infty$.

If $\inf_n\bar{u}_n>-\infty$, then, up to selecting a further subsequence, we would obtain that $u_n$ would converge uniformly on $B_r(p)$, which cannot be.

If on the other hand $\inf_n\bar{u}_n=-\infty$, then, again up to subsequences, we would obtain $u_n\to -\infty$ uniformly on $B_r(p)$ and conclude $e^{4u_n}\to 0$ uniformly on $B_r(p)$. But this would violate \eqref{eqn: concentrazione energia}. Therefore, we conclude that the blow-up points are all points of maximum of $f_0$ and the Lemma is proved. 
\end{proof}

\begin{remark}
We notice that, using the fact that the Green function $G$ for $P_{g_0}$ satisfies
\[
\left| G(x,y) - \frac{1}{8\pi^2}\log\frac{1}{|x-y|}\right|\leq C(M,g_0)\,, \;\;\;\;\; x,y\in M, x\neq y\,,
\]
and hence $G(p,y)>0$ for any $p\in M$, $y\in B_r(p)$ and suitable $r=r(p)$, we may repeat once again all the reasoning in Proposition 3.1. \cite{Malchiodi2006} and hence obtain the inequality 
\begin{equation}\label{eqn: concentrazione energia migliorata}
\liminf_n \int_{B_r(p_\infty^{(i)})} \left(f_{\lambda_n}\right)_+e^{4u_n} dV_{g_0}\geq 4\pi^2\,,
\end{equation}
which results in an improvement of \eqref{eqn: concentrazione energia}. In the following we are using this refinement.
\end{remark}

We set $M_\infty:=M\setminus \left\{p_\infty^{(1)},\cdots ,p_\infty^{(I)} \right\}$ and assume that $\inf_n\bar{u}_n=-\infty$. Therefore, by means of \eqref{eqn: convergenza locale in C 2 alfa}, we conclude that there exists a subsequence still denoted $(u_n)_n$ which converges locally uniformly to $-\infty$  on $M_\infty$, viz we obtain the first conclusion of Theorem \ref{thm: bubbling analysis}. 

If on the other hand there holds $\inf_n\bar{u}_n>-\infty$, we obtain with the help of \eqref{eqn: la media e' limitata da sopra, almeno quello...} that
\[
\sup_n |\bar{u}_n|<\infty\,.
\] 

Again by \eqref{eqn: convergenza locale in C 2 alfa} and by Schauder-type estimates (see for instance Thm 6.2.6 \cite{Morrey66}), we eventually obtain, as $n\to \infty$,that $u_n\to u_\infty$ smoothly locally in $M_\infty$. Clearly, we may also assume pointwise convergence almost everywhere and from Fatou's Lemma and \eqref{eqn: global L1 bound} we infer \begin{equation}\label{eqn: L1 bound per u infty}
\int_M |f_0|e^{4u_\infty}\,dV_{g_0} \leq \liminf_n\int_M (|f_0|+\lambda_n)e^{4u_n}\,dV_{g_0} <\infty\,.  
\end{equation}

Since now the averages of $u_n$ are bounded, we obtain that for $p\in\left(1,4/3\right)$ 
$u_n\rightharpoonup u_\infty$ in $W^{3,p}(M;g_0)$ and that $u_\infty\in W^{3,p}(M;g_0)\cap C^{\infty}(M_\infty)$ solves the equation
\begin{equation}\label{eqn: limit equation}
\begin{split}
\Delta^2_{g_0}u_\infty -& \Div_{g_0}\left(\frac{2}{3} R_{g_0}g_0 -2\mbox{Ric}_{g_0} \right)du_\infty
+2Q_{g_0}= \\
&=2f_0e^{4u_\infty} + \sum_{j=1}^I 8\pi^2 a_j \delta_{p_\infty^{(j)}} \;\;\;\; \mbox{on  } M
\end{split}
\end{equation}
in the distribution sense, where for all $1\leq j\leq I$ there holds $a_j\geq 1$ in view of \eqref{eqn: concentrazione energia}.

\begin{proposition}\label{prop: stima dei pesi}
For every $1\leq i\leq I$ there holds $1\leq a_i\leq \frac{3}{2}$.
\end{proposition}
\begin{proof}
With the help of the Green function $G$ for $P_{g_0}$ and via the related representation formula we deduce that the functions
\[
k^{(j)}(x):= 8\pi^2a_j \,G(p_{\infty}^{(j)},x)\,, \;\;\; x\in M\,,\;\; j=1,\cdots ,I
\] 
solve in the distribution sense the equations 
\[
P_{g_0}k^{(j)}=8\pi^2a_j \left(\delta_{p_\infty^{(j)}}-1\right) \;\;\; \mbox{on  } M\,.
\]
Hence, the function $w_\infty:= u_\infty-\sum_{j=1}^Ik^{(j)}$ solves distributionally the equation
\[
P_{g_0}w_\infty= -2Q_{g_0}+2f_0e^{4u_\infty} + 8\pi^2\sum_{j=1}^I a_j \;\;\; \mbox{on  } M\,,
\]
where the right hand side is in $L^1(M)$. Since we have seen that $u_\infty\in C^{\infty}(M_\infty)$, by elliptic regularity it follows $w_\infty\in C^{\infty}(M_\infty)$.

We fix $p_\infty^{(i)}$, choose normal coordinates $y\in B_\delta(0)$ around $p_\infty^{(i)}\simeq 0$ and set $\tilde{w}_\infty:=w_\infty\circ\exp$ and $\tilde{g}_0=\exp^\ast g_0$. With the help of the standard estimates for the Green function (compare \cite{Chang-Yang95}) and since $G$ is smooth outside the diagonal, we obtain that $\tilde{w}_\infty\in W^{3,p}(B_\delta(0);\tilde{g}_0)\cap C^{\infty}(\overline{B_\delta(0)}\setminus\{0\})$ with $p\in\left[1,4/3\right)$ and that it weakly solves
\[
\Delta^2_{\tilde{g}_0}\tilde{w}_\infty= \Div_{\tilde{g}_0}\left(\frac{2}{3} R_{\tilde{g}_0}\tilde{g}_0 -2\mbox{Ric}_{\tilde{g}_0} \right)d\tilde{w}_\infty-2Q_{g_0}+2\tilde{f}_0e^{4\tilde{u}_\infty} + 8\pi^2\sum_{j=1}^I a_j
\]
in $B_\delta(0)$. Notice that the right hand side of the equation above is in $L^1(B_\delta(0); \tilde{g}_0)$. 

We write $\tilde{w}_\infty= \tilde{w}_\infty^{(1)}+\tilde{w}_\infty^{(2)}$, where $\tilde{w}_\infty^{(1)}$ classically solves
\[
\left\{
\begin{array}{ll}
  \Delta^2_{\tilde{g}_0}\tilde{w}_\infty^{(1)}=0\,, & \mbox{in } B_\delta(0)\,, \\
  \tilde{w}_\infty^{(1)}=\tilde{w}_\infty\,, \; \Delta_{\R^4}\tilde{w}_\infty^{(1)}=\Delta_{\R^4}\tilde{w}_\infty\,, & \mbox{on } \partial B_\delta(0)\,.
\end{array}
\right.
\]
Therefore, with the help of Lemma 2.3 in \cite{Lin98} we infer that for any $1\leq p <\infty$, on a sufficiently small ball $B$, there holds $e^{4|\tilde{w}_\infty^{(2)}|}\in L^p(B)$. (Actually, the aforementioned Lemma has been proven for equations involving the euclidean bi-Laplacian; but, it is not difficult to generalize it to our case.)

We observe that, because $p_\infty^{(i)}\simeq 0$ is a non-degenerate maximum point of $f_0$, there holds on a sufficiently small ball that $C^{-1}|y|^2\leq |\tilde{f}_0(y)|\leq C |y|^2$ for some constant $C>1$. Hence, we conclude
\[
\begin{split}
|\tilde{f}_0(y)|e^{4\tilde{u}_\infty} &= |\tilde{f}_0(y)|e^{4\tilde{w}_\infty}e^{4\sum_{j\neq i}k^{(j)}}e^{4k^{(i)}} \\
&\leq C|y|^2\underbrace{e^{4\tilde{w}_\infty^{(1)}}}_{\in C^\infty}  e^{4\tilde{w}_\infty^{(2)}} \underbrace{e^{4\sum_{j\neq i}k^{(j)}}}_{\in C^\infty}  e^{4k^{(i)}} \\
&\leq C|y|^2e^{4\tilde{w}_\infty^{(2)}} e^{4k^{(i)}} \\
&\leq C|y|^{2-4a_1} e^{4\tilde{w}_\infty^{(2)}}
\end{split}
\]
and similarly $|\tilde{f}_0(y)|e^{4\tilde{u}_\infty}\geq C^{-1} |y|^{2-4a_1} e^{4\tilde{w}_\infty^{(2)}}$. We fix $1<q\leq 2$ and choose $p=1/(q-1)$. Therefore, on a sufficiently small ball $B$, we obtain
\[
\begin{split}
\int_B |y|^{\frac{2-4a_i}{q}}\,dy &= \int_B \left(|y|^{2-4a_i}e^{4\tilde{w}_\infty^{(2)}}\right)^{1/q}
e^{\frac{-4\tilde{w}_\infty^{(2)}}{q}} \,dy\\
&\leq \left( \int_B |y|^{2-4a_i}e^{4\tilde{w}_\infty^{(2)}} \,dy\right)^{1/q}\left(\int_Be^{\frac{-4\tilde{w}_\infty^{(2)}}{q-1}}\,dy\right)^{1-1/q}\\
&\leq C \left(\int_{\exp(B)}|f_0|e^{4u_\infty}\, dV_{g_0}  \right)^{1/q} \left(\int_Be^{\frac{4|\tilde{w}_\infty^{(2)}|}{q-1}}\,dy\right)^{1-1/q}\\
&\leq C \left(\int_Be^{4p|\tilde{w}_\infty^{(2)}|}\,dy\right)^{1-1/q} \leq C(q)\,,
\end{split}
\]
where we have used \eqref{eqn: global L1 bound}. Then, we conclude that $1\leq a_i \leq 3/2$, for $1\leq i\leq I$.

\end{proof}
\begin{proof}[Proof of Theorem \ref{thm: bubbling analysis} (completed)]    
It remains to analyse the blow-up behavior near each point $p^{(i)}_\infty$, $1\leq i\leq I$. We choose $\delta>0$ and consider the exponential map
\[
\exp: B_\delta(0)\to\exp(B_\delta(0))\,,\;\;\;\; \exp(0)=p^{(i)}_\infty
\]
($\delta$ is chosen sufficiently small in order to guarantee that in $\exp(B_\delta(0))$ the only point of maximum of $f_0$ is $p^{(i)}_\infty$). We set $K_n:={\{p\in M: f_0(p)+\lambda_n\geq 0\}}\cap \exp(B_\delta(0))$ and observe that equation \eqref{eqn: concentrazione energia migliorata} implies, up to subsequences,
\begin{equation}\label{eqn: i max esplodono nella parte positiva}
\lim_n \left(\lambda_n\max_{K_n}e^{4u_n}\right)=\infty\,.  
\end{equation}
Therefore, there exists a sequence $(p_n^{(i)})_n\subset M$ such that $u_n(p_n^{(i)})=\max_{K_n}u_n\to\infty $ and $p_n^{(i)}\to p^{(i)}_\infty$ as $n\to\infty$. To alleviate our notation, we set $p_n:=p_n^{(i)}$ and $x_n:=\exp^{-1}(p_n)\to 0$, and consider the pull-back metric $\tilde{g}_0=\exp^\ast g_0$. Therefore, by definition we have
\[
P_{\tilde{g}_0}\tilde{u}_n(x) + 2Q_{g_0} = 2\tilde{f}_{\lambda_n}(x)e^{4\tilde{u}_n(x)}\,, 
\;\;\;\; x\in B_{\delta}(0)\,,
\]
where $\tilde{u}_n=u_n\circ\exp$ and $\tilde{f_0}=f_0\circ\exp$.

Since normal coordinates are determined up to the action of the orthogonal group, we can assume from the beginning that
$\tilde{f}_0$ admits the following expansion
\begin{equation}\label{eqn: espansione di f0}
\tilde{f}_0(x)= -\sum_{i=1}^4\alpha_i x_i^2 + O(|x|^3)\,, \;\;\; 0<\alpha_1\leq\cdots \leq \alpha_4\,, \; x\in B_\delta(0)\,,  
\end{equation}
thanks to the fact the $p_\infty^{(i)}\simeq 0$ is a non-degenerate point of maximum of $f_0$. Provided that we choose $\delta$ sufficiently small from the beginning, we can further assume that
\[
-\frac{3}{2}\sum_{i=1}^4\alpha_i x_i^2 \leq\tilde{f}_0(x) \leq -\frac{1}{2}\sum_{i=1}^4\alpha_i x_i^2
\]
for all $x\in B_\delta(0)$. But then the following inclusions hold true
\begin{equation}\label{eqn: ellissoidi di controllo}
\Theta_2(n)\subset \exp^{-1}(K_n)\subset \Theta_1(n)\,,   
\end{equation}
where $\Theta_1(n)$ ( respectively $\Theta_2(n)$) is the ellipsoid of centre 0 and semi-axis of length $\sqrt{\frac{2\lambda_n}{\alpha_i}}$ ( respectively $\sqrt{\frac{2\lambda_n}{3\alpha_i}}$), $i=1,...,4$.

We first deal with the case: \\
\textbf{i)}
\[
\limsup_n \lambda_n^3e^{4\tilde{u}_n(x_n)}=\infty\,, \;\;\; \limsup_n \frac{\sqrt{\lambda_n}}{|x_n|}> \beta \sqrt{\frac{3\alpha_4}{2}}\,,
\]
where $\beta\geq 2$. 

Under these assumptions we define $r_n>0$ as
\[
r_n^4\lambda_n e^{4\tilde{u}_n(x_n)}= \frac{1}{2}\,.
\]
By \eqref{eqn: i max esplodono nella parte positiva} it immediately follows $r_n\to 0$. We now define the map
\[
\begin{split}
V_n:&\,\, x\longmapsto x_n+r_nx \\
& B_{\delta/r_n}(-x_n/r_n)\to B_\delta(0)
\end{split}
\]
and notice that $B_{\delta/r_n}(-x_n/r_n)$ exhausts $\R^4$ as $n\to\infty$. We consider the metric $\hat{g}_n=r_n^{-2}V_n^\ast\,\tilde{g}_0$ on $B_{\delta/r_n}(-x_n/r_n)$ and the functions
\[
\hat{u}_n(x)= \tilde{u}_n(V_n(x))-\tilde{u}_n(x_n)\,, \;\;\; x\in B_{\delta/r_n}(-x_n/r_n)\,.
\]
Therefore, for all $n$ sufficiently large we have $\hat{u}_n(0)=0$ and
\[
P_{\hat{g}_n}\hat{u}_n(x) + 2r_n^4Q_{g_0} = \left(\frac{\tilde{f}_0(V_n(x))}{\lambda_n} + 1\right)e^{4\hat{u}_n(x)}\,, 
\;\;\;\; x\in B_{\delta/r_n}(-x_n/r_n)\,.
\]
Furtermore, there holds for any $m\in\N_0$ that $\hat{g}_n\to\delta_{\R^4}$ in $C^m_{loc}(\R^4)$ as $n\to\infty$. By a change of variable we also obtain
\[
\begin{split}
\int_{B_{\delta/r_n}(-x_n/r_n)} e^{4\hat{u}_n(x)}\, dV_{\hat{g}_n}(x) &= \int_{B_{\delta/r_n}(-x_n/r_n)} e^{4\hat{u}_n(x)}r_n^{-4}\, dV_{V_n^\ast\,\tilde{g}_0}(x) \\
&= \int_{B_{\delta/r_n}(-x_n/r_n)} e^{4\tilde{u}_n(V_n(x))}e^{-4\tilde{u}_n(x_n)}r_n^{-4}\, dV_{V_n^\ast\,\tilde{g}_0}(x)\\&=  \int_{B_{\delta}(0)} e^{4\tilde{u}_n(x)}2\lambda_n\, dV_{\tilde{g}_0}(x) \\
&\leq 2\lambda_n \int_M e^{4u_n(x)}\, dV_{{g}_0}(x) 
\end{split}
\]
and hence, in view of \eqref{eqn: volume bound}, for any $\Omega\subset\subset\R^4$ we obtain
\begin{equation}\label{eqn: volume bound 2}
\limsup_n \int_\Omega e^{4\hat{u}_n(x)}\, dV_{\hat{g}_n}(x)\leq 64\pi^2\,.
\end{equation}
We fix $\Omega=B_R(0)$. In view of \eqref{eqn: espansione di f0} and the assumption $\limsup_n \frac{\sqrt{\lambda_n}}{|x_n|}> \beta \sqrt{\frac{3\alpha_4}{2}}$, and since $\frac{r_n^2}{\lambda_n}\to 0$, we deduce that $\frac{\tilde{f}_0(V_n(x))}{\lambda_n} + 1\to \lim_n \frac{\tilde{f}_0(x_n)}{\lambda_n}+1=:c_\infty\in \left(0,1\right]$ in $C^1(B_R(0))$.

Observe that for any $n>N(\Omega)$ we have that for any $z\in\overline{\Omega}$
\begin{equation}\label{eqn: ho spazio per le stime}
B_{r_n}(x;M) \subset\subset \exp(\Theta_2(n))\subset K_n
\end{equation}
where $x:=\exp(V_n(z))$ and $B_{r_n}(x;M)$ is the geodesic ball in $M$ of centre $x$ and radius $r_n$. Therefore, for all $n>N(\Omega)$ and such $x$ we can write with the help of the Green function for $P_{g_0}$ 
\[
\begin{split}
|\nabla^ju_n|_{g_0}(x) &\leq \int_M |\nabla^jG(x,y)|_{g_0} |f_{\lambda_n}(y)e^{4u_n(y)}-2Q_{g_0}|\,dV_{{g}_0}(y)\\
&\leq \int_M |\nabla^jG(x,y)|_{g_0} |f_{\lambda_n}(y)e^{4u_n(y)}|\,dV_{{g}_0}(y)+ C(M,g_0,j)
\end{split}
\]
with $j=1,2,3$; here we have used the estimates 
\[
|\nabla^jG(x,y)|_{g_0}\leq C(M,g_0,j)|x-y|^{-j}
\]
(compare \cite{Malchiodi2006}). We have to deal with the first term: notice that in view of \eqref{eqn: global L1 bound} we easily obtain
\[
\begin{split}
\int_{M\setminus B_{r_n}(x;M) } &|\nabla^jG(x,y)|_{g_0} |f_{\lambda_n}(y)e^{4u_n(y)}|\,dV_{{g}_0}(y)
\leq \\
& \leq  C(M,g_0,j) r_n^{-j}\int_{M\setminus B_{r_n}(x;M) } |f_{\lambda_n}(y)e^{4u_n(y)}|\,dV_{{g}_0}(y)= O(r_n^{-j})\,.
\end{split}
\] 
For the remaining part we first observe that, because of \eqref{eqn: ho spazio per le stime}, $f_{\lambda_n}$ is positive on $B_{r_n}(x;M)$ and bounded by $\lambda_ne^{4u_n(p_n)}$. Therefore, recalling the definition of $r_n$, we can write
\[
\begin{split}
\int_{ B_{r_n}(x;M) } &|\nabla^jG(x,y)|_{g_0} |f_{\lambda_n}(y)e^{4u_n(y)}|\,dV_{{g}_0}(y)
\leq \\
& \leq  C(M,g_0,j)\lambda_ne^{4u_n(p_n)} \int_{B_{r_n}(x;M) } |x-y|^{-j}\,dV_{{g}_0}(y)\\
&=C(M,g_0,j)2r_n^{-4}O(r_n^{-j+4}) =O(r_n^{-j})\,.
\end{split}
\]
In conclusion, we have showed $|\nabla^ju_n|_{g_0}(x)\leq C + O(r_n^{-j})$ for all $x:=\exp(V_n(z))$ with $z$ ranging in $\overline{\Omega}=\overline{B_R(0)}$ and $n>N(\Omega)$. Hence, recalling the definitions of $\hat{u}_n$ and $\hat{g}_n$, it is immediate to obtain $|\nabla^j\hat{u}_n|_{\hat{g}_n}(z)=r_n^j|\nabla^ju_n|_{g_0}(x)$ and thus
\[
\sup_{z\in\overline{B_R(0)}} |\nabla^j\hat{u}_n|_{\hat{g}_n}(z) <C\,, \;\;\;\; j=1,2,3
\]
uniformly in $n$. Recalling that $\hat{u}_n(0)=0$, we deduce also
\[
|\hat{u}_n(z)|=|\hat{u}_n(z)-\hat{u}_n(0)|\leq \left(\sup_{y\in\overline{B_R(0)}} |\nabla\hat{u}_n|_{\hat{g}_n}(y)\right)|z|\leq C(R)\,, \;\; z\in B_R(0)\,.
\]
This inequality and the above bounds on the derivatives of order up to 3 enables us to apply Ascoli-Arzel\`a's theorem and obtain a subsequence $(\hat{u}_n)_n$ which converges in $C^2(B_R(0))$ to some limit function $w$. Therefore, by means of Schauder's type estimates (see for instance Thm 6.4.4 \cite{Morrey66}) and recalling that, for any $m\in\N_0$, $\hat{g}_n\to\delta_{\R^4}$ in $C^m_{loc}(\R^4)$ (and thus the coefficients in the estimates do not depend on $n$) one obtains $\hat{u}_n\to w$ in $C^{4,\alpha}_{loc}(\R^4)$, where $w$ solves the equation
\[
\Delta^2_{\R^4}w = c_\infty e^{4w}\;\;\; \mbox{on } \R^4\,.
\]
with $c_\infty\in\left(0,1\right]$. Moreover, by \eqref{eqn: volume bound 2} one obtains
\[
\int_{\R^4} e^{4w}\, dx\leq 64\pi^2
\]
as well. Finally, since the image of $B_R(0)$ via the map $\exp\circ V_n$ was compactly contained in $K_n$ for all $n$ large enough (compare \eqref{eqn: ho spazio per le stime}), it follows $\hat{u}_n(z)\leq \hat{u}_n(0)=0$ for all $z\in B_R(0)$ and hence $w\leq w(0)=0$ in $\R^4$. Therefore, after replacing the expression $r_n^4\lambda_n e^{4\tilde{u}_n(x_n)}= \frac{1}{2}$ with $r_n^4\lambda_n e^{4\tilde{u}_n(x_n)}= \frac{1}{2c_\infty}$, we obtain, with a little abuse of notation, that $w$ classically solves $\Delta^2_{\R^4}w = e^{4w}$, $w\leq w(0)=0$ and $e^{4w}\in L^1(\R^4)$. From the classification of the solutions of this equation by \cite{Lin98} we obtain that either there exists $\mu>0$ such that
\[
\Delta_{\R^4}w\geq \mu \;\;\; \mbox{in } \R^4
\]
or 
\[
w(x)= -\log\left( 1+\frac{|x|^2}{4\sqrt{6}}\right)\,.
\]
We are going to rule the first alternative out. If it occured, then, similarly in the spirit to what has already been done and following \cite{Druet-Robert2005}, we could write
\[
\begin{split}
\int_{B_R(0)} &|\Delta_{\hat{g}_n}\hat{u}_n|\, dV_{\hat{g}_n}= \int_{B_{r_nR}(x_n;M)} r_n^{-2}|\Delta_{g_0}u_n|\, dV_{g_0} \\
&=2r_n^{-2} \int_{B_{r_nR}(x_n;M)}\int_M |\Delta_{g_0}G(x,y)||f_{\lambda_n}(y)e^{4u_n(y)}-Q_{g_0}|\,
dV_{g_0}(y)dV_{g_0}(x)\\
&\leq C r_n^{-2} \int_M|f_{\lambda_n}(y)e^{4u_n(y)}-Q_{g_0}|\int_{B_{r_nR}(x_n;M)}|x-y|^{-2} \,dV_{g_0}(x)dV_{g_0}(y)\\
&\leq C R^2 \int_M|f_{\lambda_n}(y)e^{4u_n(y)}-Q_{g_0}|\,dV_{g_0}(y)=O(R^2)
\end{split}
\]
as $R\to\infty$. Then in the limit for $n$ we would obtain $\mu |S^3|R^4\leq \int_{B_R(0)} |\Delta_{\R^4}w|\,dx=O(R^2)$, which for $R>>0$ is a contradiction, and therefore the second alternative must occur and we obtain alternative a) of Thm \ref{thm: bubbling analysis}.

\textbf{ii)} We now treat the case \[
\limsup_n \lambda_n^3e^{4\tilde{u}_n(x_n)}<\infty\,.
\]
We observe that from \eqref{eqn: concentrazione energia migliorata} it easily follows
$\liminf_n \lambda_n^3e^{4\tilde{u}_n(x_n)}>0$ as well. Therefore, there holds uniformly in $n$
\begin{equation}\label{eqn: stima per u_n x_n}
 |\tilde{u}_n(x_n)+\frac{3}{4}\log\lambda_n|\leq C\,.  
\end{equation}
We now define
\[
r_n^4=\frac{\lambda_n^2}{c\alpha^2_4}
\]
where $c>0$ is sufficiently large, and the map  
\[
\begin{split}
V_n:&\,\, x\longmapsto r_nx \\
& B_{\delta/r_n}(0)\to B_\delta(0)\,.
\end{split}
\]
Eventually, we consider the metric $\hat{g}_n=r_n^{-2}V_n^\ast\,\tilde{g}_0$ on $B_{\delta/r_n}(0)$ and the functions
\[
\hat{u}_n(x)= \tilde{u}_n(V_n(x))+\frac{3}{4}\log \lambda_n\,, \;\;\; x\in B_{\delta/r_n}(0)\,.
\]
Therefore, there holds
\[
P_{\hat{g}_n}\hat{u}_n(x) + 2r_n^4Q_{g_0} = \frac{2}{c\alpha_4^2}\left(\frac{\tilde{f}_0(V_n(x))}{\lambda_n} + 1\right)e^{4\hat{u}_n(x)}\,, 
\;\;\;\; x\in B_{\delta/r_n}(0)\,.
\]
We notice that for some $L>0$ we have $\{x_n/r_n\}_n\subset B_L(0)$. Moreover, from \eqref{eqn: stima per u_n x_n} we infer that for some positive constant $C$ there holds $|\hat{u}_n(x_n/r_n)|\leq C$ uniformly in $n$. As above, it can be seen that for any subsets $\Omega\subset\subset\R^4$ there holds
\begin{equation}\label{eqn: volume bound 3}
\limsup_n \int_\Omega e^{4\hat{u}_n(x)}\, dV_{\hat{g}_n}(x)\leq C\,,
\end{equation}
where $C$ is independent of $\Omega$, and that $\frac{\tilde{f}_0(V_n(x))}{\lambda_n}$ converges uniformly as $n\to\infty$ to $\frac{1}{2}D^2f_0(p_\infty^{(i)})\left[x,x\right]$. We set $h_\infty(x):=\frac{1}{c\alpha_4^2}(D^2f_0(p_\infty^{(i)})\left[x,x\right]+2)$.

From the definition of $r_n$ and reasoning as it has already been done, one obtains, again by means of estimates involving the Green function for $P_{g_0}$, that $\sup_{z\in\overline{\Omega}} |\nabla^j\hat{u}_n|_{\hat{g}_n}(z) <C$ for $j=1,2,3$ and uniformly in $n$. Finally, for any $\Omega$ containing $B_L(0)$, we obtain for
any $z\in\overline{\Omega}$
\[
\begin{split}
|\hat{u}_n(z)|&\leq  |\hat{u}_n(z)-\hat{u}_n(x_n/r_n)|+ |\hat{u}_n(x_n/r_n)| \\
&\leq \sup_{w\in\overline{\Omega}} |\nabla\hat{u}_n|_{\hat{g}_n}(w)\,\,|z-x_n/r_n| + C\leq C(\Omega)\,.
\end{split}
\]
Then, as above, we can extract by means of standard elliptic estimates a sequence $(\hat{u}_n)_n$ converging in $C^4_{loc}(\R^4)$ to a function $\tilde{w}$, which solves
\[
\Delta^2_{\R^4}\tilde{w} = h_\infty(x) e^{4\tilde{w}}\;\;\; \mbox{on } \R^4\,,
\]
with finite volume and finite total curvature
\[
\int_{\R^4} e^{4\tilde{w}}\, dx< \infty\,, \;\;\int_{\R^4} |h_\infty|e^{4\tilde{w}}\, dx< \infty\,.
\]
Therefore, setting $w:= \tilde{w}-1/4\log(c\alpha_4^2)$, we obtain alternative b) of Thm \ref{thm: bubbling analysis}.

\textbf{iii)} We finally deal with the case
\[
\limsup_n \frac{\sqrt{\lambda_n}}{|x_n|}\leq \beta \sqrt{\frac{3\alpha_4}{2}}\,,
\]
where $\beta\geq 2$. With this assumption and recalling \eqref{eqn: ellissoidi di controllo}, we deduce the existence of a costant $C\geq 1$ such that $C^{-1}\sqrt{\lambda_n}\leq |x_n|\leq C\sqrt{\lambda_n}$. We define $r_n$, the map $V_n$ and $\hat{u}_n$ in the same way as in step \textbf{ii)}. Then, it follows
\begin{equation}\label{eqn: u hat n di x_n/r_n non esplode}
\inf_n \hat{u}_n(x_n/r_n)>-\infty
\end{equation}
and $\{x_n/r_n\}_n\subset B_L(0)$ for some $L>0$.

Moreover, once again we obtain, for any $\Omega\subset\subset\R^4$, equation \eqref{eqn: volume bound 3}, uniform convergence of the $Q$-curvature to $h_\infty$ as well as $\sup_{z\in\overline{\Omega}} |\nabla^j\hat{u}_n|_{\hat{g}_n}(z) <C$ for $j=1,2,3$ and uniformly in $n$.

Now we fix $\Omega:= B_R(0)$ with $R>L$ and define 
\[
v_n(x):= \hat{u}_n(x)-\hat{u}_{n,R}\,, \;\;\; x\in B_R(0)
\]
where $\hat{u}_{n,R}:=\frac{1}{Vol(B_R(0);\hat{g}_n)}\int_{B_R(0)}\hat{u}_n \, dV_{\hat{g}_n}$. We choose $p>4/3$. Hence, via Poincar\'e's inequality, via the estimates involving the derivatives of $\hat{u}_n$, and recalling that $\hat{g}_n\to\delta_{\R^4}$ in $C^m_{loc}(\R^4)$ for any $m\geq 1$, we obtain that $(v_n)_n$ is bounded in $W^{3,p}(B_R(0),dx)$. Therefore, by reflexivity and Sobolev embedding, we obtain, up to subsequences, that $v_n\to v_\infty$ in $C^0(\overline{B_R(0)})$.

We observe that
\[
\begin{split}
C &\geq \int_{B_R(0)}e^{4\hat{u}_n} \, dV_{\hat{g}_n}=\int_{B_R(0)}e^{4v_n} \, dV_{\hat{g}_n}\exp(4\hat{u}_{n,R})\\ 
& =\left(o(1)+ \int_{B_R(0)}e^{4v_\infty} \, dx\right)\exp(4\hat{u}_{n,R})\,,
\end{split}
\]
with $o(1)\to 0 $ as $n\to\infty$. Therefore, there holds $\hat{u}_{n,R}\leq C$ uniformly in $n$. From \eqref{eqn: u hat n di x_n/r_n non esplode} and the fact $\{x_n/r_n\}_n\subset B_L(0)\subset B_R(0)$,
we also infer $\hat{u}_{n,R}\geq -C$. Hence, up to subsequences, as already done in step \textbf{ii)}, we obtain once again locally smooth convergence of
$\hat{u}_n$ to the limit function of alternative b) of Thm \ref{thm: bubbling analysis}. That completes the proof.
\end{proof}

\begin{remark}
If we couple equations \eqref{eqn: concentrazione energia migliorata} and \eqref{eqn: volume bound}, we infer that our sequence $(u_n)_n$ can blow up at at most $I=8$ points, regardless of the number of points of maximum which $f_0$ possesses. Therefore, if the function $f_0$ has more than 8 non-degenerate points of maximum, in principle one could expect that for all $0<\lambda<<1$ the functional $E_\lambda$ admits at least three different critical points.
\end{remark}

\section{Appendix}
\paragraph{\textbf{A}}
We are going to prove respectively equations (\ref{eqn: M1}), (\ref{eqn: M3}) and (\ref{eqn: M2}). Recalling (\ref{eqn: laplaciano euclideo di w lambda}), we can expand $M_1=I+II+III$ where
\[
I:= \int_{\lambda\leq |x|\leq \sqrt{\lambda}}
 |x|^{-4}\left[\delta^{-1}\xi''\left(\delta^{-1}\log\left(\frac{1}{|x|}\right)  \right) -
2\xi'\left(\delta^{-1}\log\left(\frac{1}{|x|}\right) \right)\right]^2 dx\,,
\]
\[
II:=\int_{\sqrt{\lambda}\leq |x|\leq \frac12 } 4|x|^{-4} dx\,,
\]
\[
\begin{split}
III:= & \int_{\frac12\leq |x|\leq 1} 
\left\{ 
-\tau'\left(|x| \right)|x|^{-1}\left[ 
2 +5\log\left(\frac{1}{|x|} \right) \right] \right. + \\
&\\
&-2\tau\left(|x|\right)|x|^{-2}
+ \log\left(\frac{1}{|x|} \right) \tau''\left(|x|\right)\bigg\}^2 dx\,.
\end{split}
\]
Recalling equation (\ref{eqn: stime per chi 1}) and that $\delta=\frac12\log\left(1/\lambda\right)$,
and using the abbreviations $\xi''(\cdot):=\xi''\left(\delta^{-1}\log\left(\frac{1}{|x|}\right)  \right)$ and $\xi'(\cdot):=\xi'\left(\delta^{-1}\log\left(\frac{1}{|x|}\right)  \right)$, we obtain
\[
\begin{split}
I &=  4\int_{\lambda\leq |x|\leq \sqrt{\lambda}} |x|^{-4}
\left[\frac{(\xi''(\cdot))^2}{\log^2\left(1/\lambda\right)} + (\xi'(\cdot))^2 -
\frac{2\xi''(\cdot)\xi'(\cdot)}{\log\left(1/\lambda\right)} \right]dx \\
&\leq 4 \left[\frac{||\xi''||_\infty^2}{\log^2\left(1/\lambda\right)} + A_0^2+
\frac{2A_0||\xi''||_{\infty}}{\log\left(1/\lambda\right)} \right]
\int_{\lambda\leq |x|\leq \sqrt{\lambda}} |x|^{-4}dx\\
&= 4\pi^2 \left[\frac{||\xi''||_{\infty}^2}{\log\left(1/\lambda\right)} + 2A_0||\xi''||_{\infty}+
A_0^2\log\left(1/\lambda\right) \right] 
\end{split}
\]
and
\[
II=-8\pi^2\log2+4\pi^2\log\left(1/\lambda\right)\,.
\]
For the last term we have
\[
\begin{split}
III &\leq  3\int_{\frac12\leq |x|\leq 1} 
(\tau'(|x|))^2|x|^{-2}\left[ 
2 +5\log\left(\frac{1}{|x|} \right) \right]^2  + \\
&\\
&+4\tau^2(|x|)|x|^{-4}
+ \log^2\left(\frac{1}{|x|} \right) (\tau''(|x|))^2 dx\,.
\end{split}
\]
Since the functions $r\mapsto \log\left(\frac{1}{r}\right)$ and $r\mapsto\log^2\left(\frac{1}{r}\right)$ are monotone decreasing and positive on the interval $(\frac{1}{2},1)$, we obtain
\[
III\lesssim \int_{\frac12\leq |x|\leq 1}
\left(|x|^{-2} +|x|^{-4}\right)dx +C \leq C
\]
uniformly in $\lambda$. Hence, choosing a smaller $\lambda^{\varepsilon}$ if necessary, we infer that for all $0<\lambda<\lambda^{\varepsilon}$ 
\[
M_1\leq 4\pi^2(A_0^2+1)\log\left(1/\lambda\right) + C_0\,,
\]
where $C_0$ depends at most quadratically on the supremum norm of $\xi''$ but it does not depend on $\lambda$.

In order to obtain equation (\ref{eqn: M3}), we note that, in view of (\ref{eqn: gradiente euclideo di w lambda}), the supremum norm of the radial derivative $z_\lambda'$ on the ball $B_1(0) $ is of the order $O(\lambda^{-1}) $ as $\lambda \downarrow 0$. Therefore, it holds
\[
||z_\lambda' O''(r^{N-1})||_{\infty}=O(\lambda^{\frac{N-3}{2}})\,,
\] 
which leads to
\[
M_3=O(\lambda^{N-3})\,.
\]
Finally, combining the estimates about $M_1$ and $M_3$ and by H\"older's inequality, we obtain
\[
\begin{split}
|M_2| &\leq 2\int_{B_1(0)}\left| 
\Delta_{\R^4}z_\lambda \, z_\lambda' O''(r^{N-1})\right|dx\\
&\leq 2||\Delta_{\R^4}z_\lambda ||_{L^1(B_1(0))} \,||z_\lambda' O''(r^{N-1})||_{\infty}=O(\lambda^{\frac{N-3}{2}})
\end{split} 
\]
as $\lambda\downarrow 0$. 

\paragraph{\textbf{B}} We are going to prove \eqref{eqn: precompattezza 2}. Starting from the last two lines
of \eqref{eqn: precompattezza 1}, we obtain respectively:
\[
\begin{split}
\int_M |v_n\Delta_{g_0}v_n|\,|\nabla_{g_0}\eta|^2_{g_0}\,dV_{g_0} &\leq
||\nabla_{g_0}\eta||^2_\infty ||\Delta_{g_0}v_n ||_{L^t(M)}||v_n ||_{L^{t'}(M)}\\
&\leq C(\eta,d), \;\;\; t\in(1,2),\; t'=\frac{t}{t-1}\,;
\end{split}
\]
\[
\int_M \left(v_n\Delta_{g_0}\eta\right)^2 \,dV_{g_0}\leq ||\Delta_{g_0}\eta||^2_\infty
||v_n||^2_{L^2(M)}\leq C(\eta,d)\,;
\]
\[
\int_M \left|v_n^2 D\left(\nabla_{g_0}\eta,\nabla_{g_0}\eta\right)\right|dV_{g_0}\leq
c(M,g_0)||\nabla_{g_0}\eta||^2_\infty ||v_n||^2_{L^2(M)}\leq C(\eta,d)\,;
\]
\[
\begin{split}
\int_M \left|v_n\Delta_{g_0}\eta\, g_0\left(\nabla_{g_0}v_n,\nabla_{g_0}\eta\right)\right| dV_{g_0} &\leq
||\Delta_{g_0}\eta||_\infty||\nabla_{g_0}\eta||_\infty ||\nabla_{g_0}v_n ||_{L^t(M)}||v_n ||_{L^{t'}(M)}\\
&\leq C(\eta,d), \;\;\; t\in(1,4),\; t'=\frac{t}{t-1}\,;
\end{split}
\]
\[
\int_M g_0\left(\nabla_{g_0}v_n,\nabla_{g_0}\eta\right)^2  dV_{g_0}\leq ||\nabla_{g_0}\eta||_\infty^2
||\nabla_{g_0}v_n ||^2_{L^2(M)}\leq C(\eta,d)\,;
\]
the claim follows.
\paragraph{\textbf{C}}
We are going to prove the following:
\begin{proposition}\label{prop: Palais-Smale}
Let $(M,g_0)$ be closed and connected with $k_P<0$, $P_{g_0}\geq 0$ and $\ker(P_{g_0})=\left\{constants\right\}.$ Let $0\neq f\in C^2(M)$. Then the functional \eqref{eqn: the energy} satisfies the Palais-Smale condition at any level $\beta\in\R$.
\end{proposition}
\begin{proof}
Let $(u_k)_k\subset \Htwo$ be a Palais-Smale sequence at the level $\beta$ for the functional $E_f$, viz. as $k\to\infty$
\[
E_f(u_k)= \langle P_{g_0}u_k,u_k \rangle +4Q_{g_0}\int_M u_k\, dV_{g_0} - \int_M fe^{4u_k}dV_{g_0}\to\beta
\]
and
\[
||DE_f(u_k)||\to 0\,.
\]
Therefore, since in particular $DE_f(u_k)[1]\to 0$, we obtain $\int_M fe^{4u_k}dV_{g_0}\to k_P$ and
\begin{equation}\label{eqn: Palais-Smale}
\langle P_{g_0}u_k,u_k \rangle +4Q_{g_0}\int_M u_k\, dV_{g_0}=\beta+k_P +o(1)  
\end{equation}
as $k\to\infty$.

\texttt{Claim:} $\sup_k\int_M u_k \, dV_{g_0}<\infty$.

We argue by contradiction and assume that there exists a subsequence still denoted $u_k$ such that
$\lim_k\int_M u_k \, dV_{g_0}=\infty$. Hence, by H\"older inequality it follows that also the quantity
$||u_k||_{L^2(M)}$ tend to $\infty$. We set
\[
v_k:=\frac{u_k}{||u_k||_{L^2(M)}}\,.
\]
Trivially, $||v_k||_{L^2(M)}=1$ and from \eqref{eqn: Palais-Smale} we obtain
\begin{eqnarray*}
 \langle P_{g_0}v_k,v_k \rangle &=& \langle P_{g_0}u_k,u_k \rangle \,\,||u_k||_{L^2(M)}^{-2}\\
 &=&  \left(-4Q_{g_0}\int_M u_k\, dV_{g_0}+\beta+k_P +o(1)\right)\,\,||u_k||_{L^2(M)}^{-2}\,.
\end{eqnarray*}
Because $\left|-4Q_{g_0}\int_M u_k\, dV_{g_0}\right|\,\,||u_k||_{L^2(M)}^{-2}\leq -4Q_{g_0}||u_k||_{L^2(M)}^{-1}$, it follows that the right hand side of the expression above tends to zero as $k\to\infty$ and consequently
\[
\langle P_{g_0}v_k,v_k \rangle\to 0\,.
\]
Therefore, up to subsequences, we can assume $v_k\rightharpoonup v$ in $\Htwo$ and $v_k\to v$ in $L^2(M)$. By means of Poincar\'e's inequality, we infer $v\equiv c\in\{-1,1\}$. On the other hand,
\[
\frac{\int_M u_k \, dV_{g_0}}{||u_k||_{L^2(M)}}=\int_M v_k \, dV_{g_0}\to c
\]
and, because by assumption $\lim_k\int_M u_k \, dV_{g_0}=\infty$, we deduce $v\equiv 1$. We define $\phi_k=\frac{f}{||u_k||_{L^2(M)}}\in\Htwo$. Obviously, $\phi_k\to 0$ in $\Htwo$ and so it follows
\[
\langle P_{g_0}u_k,\phi_k \rangle +2Q_{g_0}\int_M \phi_k\, dV_{g_0} - 2\int_M f\phi_ke^{4u_k}dV_{g_0}\to 0
\]
or, equivalently,
\[
\langle P_{g_0}v_k,f \rangle +2Q_{g_0}\int_M f\, dV_{g_0}\,\,||u_k||_{L^2(M)}^{-1} - 2\int_M 
\frac{f^2e^{4u_k}}{||u_k||_{L^2(M)}} dV_{g_0}\to 0
\]
From above, it thus follows $\int_M \frac{f^2e^{4u_k}}{||u_k||_{L^2(M)}} dV_{g_0}\to 0$. On the other hand, we have
\[
\int_M \frac{f^2e^{4u_k}}{||u_k||_{L^2(M)}} dV_{g_0} \geq \int_M \frac{4f^2u_k}{||u_k||_{L^2(M)}} dV_{g_0} =\int_M 4f^2v_k dV_{g_0}\to \int_M 4f^2 dV_{g_0} >0
\]
since $v_k\to 1$ in $L^2(M)$. The contradiction proves the claim.

Equation \eqref{eqn: Palais-Smale} also implies $4Q_{g_0}\int_M u_k\, dV_{g_0}\leq\beta+k_P +o(1) $ and therefore $\inf_k\int_M u_k \, dV_{g_0}>-\infty$ and
\[
\sup_k\left|\int_M u_k \, dV_{g_0}\right|<\infty\,.
\]
Again by \eqref{eqn: Palais-Smale} it follows $\sup_k\langle P_{g_0}u_k,u_k \rangle<\infty$ and by Poincar\'e's inequality we conclude that $(u_k)_k$ is bounded in $\Htwo$. 

From this fact, it is now standard to extract from $(u_k)_k$ a converging subsequence in $\Htwo$. That concludes the proof.
\end{proof}

\section*{Acknowledgments}
I would like to thank Michael Struwe for the useful suggestions aimed to improve the presentation of this paper.

\end{document}